\documentclass[11pt,reqno]{amsart}
\usepackage[utf8]{inputenc}
\usepackage[T1]{fontenc}
\usepackage{amsmath}
\usepackage{amsfonts}
\usepackage{amssymb}
\usepackage{amsthm}
\usepackage{setspace}
\usepackage{enumerate}
\usepackage{enumitem}
\usepackage{todonotes}
\usepackage{geometry}
\usepackage{graphicx,color}
\usepackage[bottom]{footmisc}
\usepackage[english]{babel}
\usepackage[autostyle, english = american]{csquotes}
\usepackage[unicode=true,
 bookmarks=true,bookmarksnumbered=false,bookmarksopen=false,
 breaklinks=false,pdfborder={0 0 1},backref=false,colorlinks=true]{hyperref}

\newcounter{theorem}[section]
\numberwithin{theorem}{section}

\newtheorem{definition}[theorem]{Definition}
\newtheorem{theo}[theorem]{Theorem}
\newtheorem{lemma}[theorem]{Lemma}
\newtheorem{prop}[theorem]{Proposition}
\newtheorem{corollary}[theorem]{Corollary}

\newtheorem*{remark}{Remark}

\newcommand\norm[1]{\left\lVert#1\right\rVert}

\DeclareMathOperator{\vol}{vol}

\begin{document}

	\title[The curvature of $P(S^1)$]
	{The Wasserstein space of the circle is not flat}

	\author[A. M. de S\'a Gomes]{Andr\'e Magalh\~aes de S\'a Gomes}
	\address{Andr\'e Magalh\~aes de S\'a Gomes\\
		Instituto de Matemática e Estatística - USP\\
		Rua do Matão, 1010 | 05508-090, São Paulo, SP, Brazil.}
	\email{andremsg@usp.br}

	\author[C.S. Rodrigues]{Christian S.~Rodrigues}
	\address{Christian S.~Rodrigues\\
		Institute of Mathematics, Department of Applied Mathematics\\
		Universidade Estadual de Campinas\\
		13.083-859 Campinas - SP\\
   		Brazil\\
   		and Max-Planck-Institute for Mathematics in the Sciences\\
  		Inselstr. 22\\
  		04103 Leipzig\\
  		Germany
		}
  	\email{rodrigues@ime.unicamp.br}

\author[L. A. B. San Martin]{Luiz A. B. San Martin}
	\address{Luiz A. B. San Martin\\
		Institute of Mathematics, Department of Mathematics\\
		Universidade Estadual de Campinas\\
		13.083-859 Campinas - SP\\
   		Brazil
		}
  	\email{smartin@ime.unicamp.br}

	\date{\today}

	\begin{abstract}
		We show that the Wasserstein space of the circle is not flat, in contrast with the flatness of the Wasserstein spaces of the real line and of an interval. Its sectional curvature is governed by a symplectic form on the Fourier coefficients, so that a plane is flat exactly when it is isotropic for that form, and on the plane spanned by the $n$-th Fourier modes the curvature is bounded above by $3n^2$, with equality exactly at the uniform measure. It is therefore unbounded. We compute it inside the formal Riemannian structure of Otto and confirm it directly from the definition of $W_2$, which also shows that no neighbourhood of the uniform measure in $P_2(S^1)$ is CAT(0). The result corrects the widely stated expectation that a Wasserstein space is flat whenever its base manifold is one-dimensional.
		
		The computation is carried out in the global frame of the tangent bundle of $P^\infty(M)$ induced by the eigenfunctions of the Laplacian of $M$, whose existence was observed by Lott. We prove that it is a Schauder basis of every tangent space and write Otto's metric and its Levi-Civita connection in it. On the circle this gives the metric and connection coefficients of $P^\infty(S^1)$ in closed form, as finite combinations of the Fourier coefficients of the density. A final section puts the frame to work on a computation from dynamics, reading the derivative of the push-forward action of an expanding circle map at its invariant measure as a weighted shift on the Fourier frequencies.
	\end{abstract}

	%\keywords{Disintegration Theorem, Probability, Optimal Transport, Ergodic Theory, Dynamical Systems}

	\subjclass{49Q22, 58B20 (primary), 53C21, 53C22, 37E10 (secondary)}

	\maketitle
	
	\tableofcontents

%%%%%%%%%%%%%%%%%%%%%%%%%%%%%%%%%%%%%%%%%%%%%%%%%%%%%%%%%%%%%%%%%%%%%%%%%%%%%%%%%%%%%%%%%%%%%%%%%%%%%%%%%%%%%%%%%%

	\section*{Introduction}

It is a classical fact that the Wasserstein space of the real line, and that of an interval \cite{chodosh2011interval}, are flat, and this is often extended to the assertion that a Wasserstein space is flat whenever its base manifold is one-dimensional. In a recent paper on sectional curvature, for instance, Sturm recalls ``the well-known fact that the sectional curvature in the Kantorovich-Wasserstein space $(P_2(M),W_2)$ vanishes whenever the base space $M$ is 1-dimensional'' \cite[Remark 5.9]{sturm2026sectional}. The purpose of this paper is to show that the circle is a counterexample. The space $P_2(S^1)$ is not flat, its sectional curvature along the plane spanned by the two Fourier modes of frequency $n$ equals $3n^2$ at the uniform measure, and is therefore unbounded as $n$ grows. We compute it inside the formal Riemannian structure of Otto, and confirm it directly from the definition of $W_2$.

The Wasserstein space over a closed Riemannian manifold carries a rich geometric structure, formally described by the Otto calculus as an infinite-dimensional Riemannian manifold. That description was introduced in order to read dissipative evolution equations as gradient flows of entropy functionals \cite{otto2001geometry}, and its content is variational throughout: geodesics are displacement interpolations \cite{mccann1997aconvexity}, and the convexity of entropy along them characterizes lower Ricci bounds on the base \cite{lott2009ricci,sturm2006geometry}. The two objects computed here, the geodesic equation and the sectional curvature, belong to that circle. 

Several approaches to the geometry of Wasserstein spaces have been developed in the literature. The formal Riemannian structure introduced by Otto \cite{otto2001geometry} has been further investigated and placed on rigorous grounds in subsequent works \cite{lott2006some, villani2009optimal, ambrosio2005gradient}, among others. Both extrinsic and intrinsic viewpoints have been explored, leading to important insights into geodesics, curvature, and variational structures \cite{ding2021geometry,GangboKimPacini2011}.

Despite these advances, the Levi-Civita connection of $P^\infty(M)$ has been described through formulas for covariant derivatives of vector fields rather than through its coefficients in a frame, and to the best of our knowledge those coefficients have not been computed for any specific base manifold.

This gap is sharpest in contrast with the differential structure developed by Gigli \cite{gigli2011inverse,ambrosiogigli2008}. There the tangent and cotangent objects of $P_2(M)$, the Levi-Civita connection and the parallel transport are characterized variationally, by the properties they are required to satisfy rather than by formulas, and it is precisely this that makes the theory available over measures that need not be absolutely continuous and without any smoothness assumption. The price is that the objects so obtained are implicit. Parallel transport, for instance, is produced as a limit of approximations along a geodesic, see also the intrinsic construction of \cite{lott2017intrinsic}, and no expression results from which one could read off a connection coefficient or integrate a geodesic equation. The metric measure version of that calculus, in which vector and tensor fields are elements of $L^\infty$-modules and the connection is defined at that level, is unavailable here for a more basic reason: it presupposes a reference measure on the underlying space and a curvature-dimension condition in the sense of \cite{sturm2006geometry,lott2009ricci}, neither of which $P_2(M)$ provides. The trade we make here is the opposite one, generality for explicitness.

The tool that gives access to the curvature is a global frame for the tangent bundle of $P^\infty(M)$. Lott observed that a basis $\{\phi_i\}$ of $C^\infty(M)$ modulo constants induces a global basis $\{V_{\phi_i}\}$ of $TP^\infty(M)$ \cite[Section 4]{lott2006some}. We give a proof when the $\phi_i$ are eigenfunctions of the Laplace-Beltrami operator, and we make precise the sense in which the family is a basis at each measure (Theorem \ref{thm:parallelization}). In this frame Otto's metric and its Levi-Civita connection have coefficients, and on the circle these coefficients are finite combinations of the Fourier coefficients of the density. Along the way we recover, in a unified and intrinsic notation, several differential and Riemannian formulas that have previously been derived through extrinsic or formal arguments.

We carry out the computation in the first case of interest, the unit circle, where the Fourier basis provides the frame, and we determine the metric coefficients, the connection coefficients and the curvature of $P^\infty(S^1)$.

The curvature turns out to be governed by a symplectic form on the Fourier coefficients: a plane is flat exactly when it is isotropic for that form, and on the plane spanned by the $n$-th modes the sectional curvature is at most $3n^2$, with equality exactly at the uniform measure. In particular $P^\infty(S^1)$ has non-negative but unbounded curvature, and is not flat.

This last conclusion deserves comment, since it runs against the expectation recalled above. It does not contradict the computations of Sturm on the flat torus, however. For $n=1$ he finds that every plane spanned by two Fourier modes of different frequencies is flat at the uniform measure \cite[Example 6.22]{sturm2026sectional}, and these are exactly the planes that are isotropic for the symplectic form above. The planes spanned by $\cos k\theta$ and $\sin k\theta$, which are the planes of $e_k$ and $e_{-k}$ in the notation of \cite{sturm2026sectional} and form the case excluded there, are the ones on which the curvature does not vanish. The reason is topological rather than analytic: for a one-dimensional base the orthogonal complement of the exact one-forms in $\Omega^1_{L^2}(M,\mu)$ is at most a line, and it is a line precisely when $H^1(M)$ does not vanish. Metrically, the same dichotomy appears as the fact that on the circle the monotone coupling is determined only up to a rotation, and optimizing that rotation produces a fourth order gain in the transport cost which is absent on the line. We carry out that computation in Proposition \ref{prop:metric-curvature}, where the fourth order defect of $W_2$ is computed from the definition of $W_2$ alone and agrees with the value given by the Otto calculus. As a purely metric consequence, no neighbourhood of the uniform measure in $P_2(S^1)$ is CAT(0). In dimension two or more the complement is infinite-dimensional for reasons unrelated to topology, so the circle is the smallest case in which this mechanism is visible.

The framework also simplifies computations in dynamical settings. In \cite{Klo13}, Kloeckner computes the derivative of the push-forward action of an expanding circle map at its absolutely continuous invariant measure, using the differential structure of Gigli. In the Fourier frame that derivative becomes a weighted shift supported on a sublattice of frequencies, a form in which the infinite multiplicity of its eigenvalues, one of the main points of \cite{Klo13}, is immediate. This is carried out in Section \ref{sec:kloeckner}.

The paper is organized as follows. Section \ref{sec:preliminaries} collects the material from optimal transport and from the differential and Riemannian calculus on Wasserstein spaces that the rest of the paper uses. Section \ref{secparallelization} contains the parallelization theorem, the metric coefficients, the Christoffel symbols, the geodesic equation and the curvature formula in the general setting. Section \ref{secExample} specializes everything to the circle. Section \ref{sec:kloeckner} reads in that frame a derivative computed by Kloeckner in \cite{Klo13}.

%%%%%%%%%%%%%%%%%%%%%%%%%%%%%%%%%%%%%%%%%%%%%%%%%%%%%%%%%%%%%%%%%%%%%%%%%%%%%%%%%%%%%%%%%%%%%%%%%%%%%%%%%%%%%%%%%%%%%%%%%%%%%%%%%%%%%%%%%%%%%%%%%%%%%%%%%%

\section{Preliminaries}\label{sec:preliminaries}

This section collects the material on which the rest of the paper is built. We first recall the basic notions of optimal transport theory: the Wasserstein distances, their relation to the weak topology on $P(M)$, and the description of geodesics of $P_2(M)$ by displacement interpolations. We then set up the differential calculus on Wasserstein spaces, defining tangent spaces through the continuity equation and introducing the weighted differential operators $\operatorname{div}_\mu$, $\Delta_\mu$ and $G_\mu$ that make the tangent structure explicit. We close with the Riemannian structure of $P^{ac}(M)$: the metric of Otto, the Lie bracket of constant vector fields and the Levi-Civita connection. The material in this section is well known, but it is scattered across sources that use different conventions and different degrees of generality. We therefore present it in a uniform notation, including proofs whenever our formulation differs from the ones available in the literature.

%%%%%%%%%%%%%%%%%%%%%%%%%%%%%%%%%%%%%%%%%%%%%%%%%%%%%%%%%%%%%%%%%%%%%%%%%%%%%%%%%%%%%%%%%%%%%%%%%%%%%%%%%%%%%%%%%%%%%%%%%%%%

\subsection{The Wasserstein Spaces}\label{secwasserstein}

Let $M$ be a connected, orientable, closed Riemannian manifold with geodesic distance $d$ and volume measure $\vol$, normalized so that $\vol(M)=1$. Let $P(M)$ denote the space of probability measures on $M$ with the weak topology \cite{billingsley2013convergence}, i.e., we say that a sequence $\mu_n$ in $P(M)$ converges to $\mu\in P(M)$ if for every continuous function $f:M\to \mathbb{R}$
$$\int_Mf d\mu_n\xrightarrow{n\to \infty}\int_M fd\mu.$$

If $T:X\to Y$ is a measurable map between measurable spaces $X$ and $Y$, and $\mu$ is a measure on $X$, we define the \textbf{push-forward} of $\mu$ by $T$ as the measure on $Y$ characterized by
$$T_*\mu(A)=\mu(T^{-1}(A)),$$
for every measurable set $A\subset Y$.

Given two measures $\mu,\nu\in P(M)$, we say that $\pi\in P(M\times M)$ is a \textbf{coupling} of $\mu$ and $\nu$ if
$$(\operatorname{proj}_1)_*\pi=\mu \quad \textrm{and} \quad (\operatorname{proj}_2)_*\pi=\nu; $$
with $\operatorname{proj}_1,\operatorname{proj}_2:M^2\to M$ being the natural projections $(x,y)\mapsto x$ and $(x,y)\mapsto y$, respectively.
We denote by $\Pi(\mu,\nu)$ the set of all couplings of $\mu$ and $\nu$, which is non-empty since it contains the product measure $\mu\times\nu$.

Another coupling of fundamental importance to us is the \textit{optimal coupling}, which is given by the following Monge-Kantorovich Theorem (see for instance \cite[Theorem 4.1]{villani2009optimal}). Intuitively, it minimizes the cost of transporting mass, measured in terms of the geodesic distance. It is therefore the tool that relates the geometry of $M$ to the length space structure of $P(M)$ and, consequently, to its geometry.

\begin{theo}\textbf{(Monge-Kantorovich)}
	Take a pair $\mu,\nu\in P(M)$ and let $p\geq 1$. Then the infimum
$$C_p(\mu,\nu):=\inf_{\pi\in\Pi(\mu,\nu)}\int_{M\times M}d^p(x,y)\, d\pi(x,y)$$
is attained. Any $\pi\in\Pi(\mu,\nu)$ attaining it is called an \textbf{optimal coupling} between $\mu$ and $\nu$.
\end{theo}

\begin{definition}
    The \textbf{Wasserstein distance} of order $p\geq 1$ between two probability measures $\mu$ and $\nu$ on $M$ is defined by
\begin{equation}\label{Wp}
	W_p(\mu,\nu):=\left(\inf_{\pi\in\Pi(\mu,\nu)}\int_{M^2}d(x,y)^p\, d\pi(x,y)\right)^{1/p}=C_p(\mu,\nu)^{1/p}.
\end{equation}
\end{definition}

Straightforward calculations show that $W_p$ is positive definite, symmetric and satisfies the triangle inequality. Moreover, since $M$ is closed its diameter is finite, so $W_p(\mu,\nu)\leq\operatorname{diam}(M)<+\infty$ for every pair $\mu,\nu\in P(M)$ and every $p\geq1$. Thus (\ref{Wp}) indeed defines a distance on all of $P(M)$, and the set underlying the Wasserstein space of order $p$ is $P(M)$ itself. We call $(P(M),W_p)$ the \textbf{Wasserstein space} of order $p$ on $M$, and write $P_p(M)$ for it whenever the metric is relevant. As the next theorem shows, $W_p$ is indeed a metrization of the weak topology.

\begin{theo}\cite[Theorem 6.9]{villani2009optimal}
    If $(\mu_k)_{k\in\mathbb{N}}$ is a sequence in ${P}_p(M)$ and $\mu$ is a measure in ${P}(M)$, then $(\mu_k)_{k\in\mathbb{N}}$ converges to $\mu$ in the weak topology if and only if $W_p(\mu_k,\mu)\xrightarrow[]{k\to\infty}0$.
\end{theo}

The space $P_p(M)$ is moreover Polish, that is, separable and completely metrizable \cite[Theorem 6.18]{villani2009optimal}.

%%%%%%%%%%%%%%%%%%%%%%%%%%%%%%%%%%%%%%%%%%%%%%%%%%%%%%%%%%%%%%%%%%%%%%%%%%%%%%%%%%%%%%%%%%%%%%%%%%%%

\subsubsection{Geodesics and displacement interpolations}

We now turn to the geodesic structure of $P_2(M)$, whose geodesics are described by the displacement interpolations of optimal transport. We recall their construction, first in the general setting of transport plans and then, for absolutely continuous measures, via transport maps.

\textbf{Displacement interpolations} are curves given by the push-forwards
$$\mu_t:=(e_t)_*\tau;$$
where $\mu_0,\mu_1\in P_2(M)$ are fixed, $e_t:C([0,1],M)\to M$ is the evaluation map $\gamma\mapsto \gamma_t$, and $\tau$ is a probability measure concentrated on the set of minimizing geodesics of $M$ such that $(e_0,e_1)_*\tau$ is an optimal coupling between $\mu_0$ and $\mu_1$. See \cite[chapter 7]{villani2009optimal} for more details. In particular, any two measures in $P_2(M)$ are joined by a displacement interpolation \cite[Corollary 7.22]{villani2009optimal}.

Following McCann's description of the expansion of a gas \cite{mccann1997aconvexity}, one may read $\mu_t$ as a gas whose particles travel along geodesics of $M$, with $\tau$ the law of those trajectories. For this reason $\tau$ is also called a \emph{random geodesic} on $M$. 

Since tangent spaces and cones at singular measures are not simple to describe---see \cite{lott2017ontangent} and \cite[Appendix A]{lott2009ricci}---let alone to compute with, we hereafter restrict our analysis to the space $P^{ac}(M)$ of measures absolutely continuous with respect to $\vol$. From Subsection \ref{sec1} onwards we restrict further to the dense subspace
$$P^{\infty}(M)=\left\{\rho\vol\colon\, \rho\in C^\infty(M),\ \rho> 0,\ \int_M\rho\,\vol=1\right\},$$
whose positivity and smoothness of the density are what the constructions below require. In the absolutely continuous setting there is a better characterization of displacement interpolations, for which the following definition is needed.

\begin{definition} 

A function $\psi:M\to\mathbb{R}\cup\{+\infty\}$ is said to be \textbf{$d^2/2$-convex} if it is not identically $+\infty$ and there is a function $\phi:M\to\mathbb{R}\cup\{\pm\infty\}$ such that for every $x\in M$, $\psi(x)=\sup_{y\in M}\{\phi(y)-d^2(x,y)/2\}$.

\end{definition}

Given $\mu_0,\mu_1\in P^{ac}(M)$, there is a \emph{$d^2/2$-convex} function $\psi$, finite and differentiable $\mu_0$-almost everywhere, such that the displacement interpolation between them is given by
$$\mu_t=(T_t)_*\mu_0,\qquad T_t(x)=\exp_x(t\nabla \psi(x)),$$
where $\nabla\psi$ denotes the gradient of $\psi$.

Furthermore, by \cite[Theorem 7.21]{villani2009optimal}, the displacement interpolations from $\mu_0$ to $\mu_1$ are precisely the curves minimizing the action
$$\mathbb{A}(\mu_t)=\sup_{n\in\mathbb{N}}\,\sup_{0=t_0<t_1<\cdots<t_n=1}\sum_{i=0}^{n-1}\frac{W_2(\mu_{t_i},\mu_{t_{i+1}})^2}{t_{i+1}-t_i}$$
among all curves in $P_2(M)$ joining $\mu_0$ to $\mu_1$, and the minimum equals $W_2(\mu_0,\mu_1)^2$. In particular they are constant speed geodesics of $(P_2(M),W_2)$.

This is the structure on which the differential calculus of the next subsection is built.

%%%%%%%%%%%%%%%%%%%%%%%%%%%%%%%%%%%%%%%%%%%%%%%%%%%%%%%%%%%%%%%%%%%%%%%%%%%%%%%%%%%%%%%%%%%%%%%%%%%%%%%%%%%%%%%%%%%%%%%%%%%%%%%%%%%%%%%%%%%%%%%%%%%%%%%%%%%%%%%%%%%%%%%%%%%%%%%%%%%%%%%%%%%%%%%%%%%%%%%%%%%%%%%%%%%%%%%%%%%%

\subsection{Differential Calculus on Wasserstein spaces}\label{sec1}

Here we fix the intrinsic framework for the differential geometry of Wasserstein spaces that will be used in the sequel. We recall the definition of tangent spaces via derivatives of absolutely continuous curves and the associated continuity equation, which links curves of measures to vector fields on the underlying manifold $M$. We then use the $L^2(\mu)$ structure to define differential operators on $M$ that describe the tangent space of $P(M)$ at $\mu$, in a similar manner to what was done in \cite{GangboKimPacini2011} for $P(\mathbb{R}^n)$, and we single out the class of curves with constant velocity field.

%In this section we introduce a compendium on differential calculus on Wasserstein spaces of Riemannian manifolds.

%%%%%%%%%%%%%%%%%%%%%%%%%%%%%%%%%%%%%%%%%%%%%%%%%%%%%%%%%%%%%%%%

\subsubsection*{Tangent Vectors and Derivatives of Curves}

As mentioned in the introduction, $P(M)$ carries a formal Riemannian structure whose length structure is the one just presented. This structure was introduced formally by Otto \cite{otto2001geometry}, and Lott and Villani made it rigorous by computing the tangent cones of $P_2(M)$ at absolutely continuous measures \cite[Appendix A]{lott2009ricci}. The tangent spaces are also described in \cite[Chapter 8]{ambrosio2005gradient}.

Here we present these tangent spaces and their notions of derivatives, with an approach based on \cite{lott2006some} and \cite{ding2021geometry}. 

A curve $\mu_t$ on $P(M)$ is said to be \emph{absolutely continuous in $L^2$} if there is $k\in L^2([0,1])$ such that $W_2(\mu_{t_1},\mu_{t_2})\leq \int_{t_1}^{t_2} k(s)\,ds$ for $t_1<t_2$. This is not to be confused with absolute continuity of a measure with respect to another. Ambrosio \cite[Theorem 8.3.1]{ambrosio2005gradient} defined a notion of derivative for such curves, which we state below.

\begin{theo}\label{theo1.1}
    Let $\mu_t$ be an absolutely continuous curve on $P(M)$, with $\mu_0$ absolutely continuous with respect to the volume measure. Then there exists a Borel vector field $Z_t$ on $M$ such that
    $$\int_0^1\int_M\norm{Z_t(x)}^2_{T_xM}d\mu_tdt<+\infty$$
    and the following conservation of mass formula 
    \begin{equation}
        \label{conservation} \dot{\mu_t}+\nabla\cdot (\mu_t Z_t)=0
    \end{equation}
    holds in the sense of distributions. Uniqueness to (\ref{conservation}) holds if moreover $Z_t$ is imposed to be in
    $$\overline{\{\nabla\psi,\,\psi\in C^\infty(M)\}}^{L^2(\mu_t)}.$$
    
\end{theo}

By sense of distributions in (\ref{conservation}) we mean that for every compactly supported $\varphi\in C^1(M)$,
$$\int_M\varphi\nabla\cdot(\mu Z_t)=-\int_M\langle Z_t, \nabla\varphi\rangle d\mu.$$

So, a differentiable curve $\mu_t$ on $P(M)$ may be written, in the sense of distributions, as
$$\mu_t=\mu_0+t\dot{\mu}_0+o(t).$$

We say that $\dot{\mu}=-\nabla\cdot(\mu_tZ_t)$ is a tangent vector of $P(M)$ at $\mu$, and we define the \textbf{tangent space} of $P(M)$ at $\mu$, denoted $T_\mu P(M)$, as the set of all such tangent vectors.

We now see how a tangent vector $X\in T_\mu P(M)$ acts via the G\^ateaux derivative on functions $F$ by
\begin{equation}\label{directionalderivative}
    (X\cdot F)(\mu)=\left.\frac{d}{dt}\right|_{t=0}F\left(\mu_t\right),
\end{equation}
with $\mu_t$ being any curve in $P(M)$ such that $\mu_0=\mu$ and $\dot{\mu}_0=X$; the expansion $\mu_t=\mu_0+t\dot{\mu}_0+o(t)$ shows that the derivative does not depend on the choice of $\mu_t$. When it exists, we call $X\cdot F(\mu)$ the \textbf{directional derivative} of $F$ at $\mu$ along $X$, and we say that $F$ is \textbf{differentiable} at $\mu$ if it has directional derivatives along every $X\in T_\mu P(M)$.

\subsubsection*{Differential Operators}

As expected, since $\nabla\cdot(\mu\nabla\psi)$ is the divergence of a gradient, it can also be interpreted as a \textbf{Laplace operator} in the $\mu$ weighted $L^2$-space, namely $L^2(M,\mu)$, and in the $L^2$-forms weighted space $\Omega^1_{L^2}(M,\mu)$. Indeed, let $d$ be the usual differential on functions and $d_\mu^*$ be its formal adjoint, then
  $$\nabla\cdot(\mu\nabla\psi)=-(d_\mu^*d\psi)\mu=(\Delta_\mu \psi )\mu;$$
 in which $\Delta_\mu=-d^*_\mu d$ on smooth functions. With this convention $\Delta_\mu$ is the negative semi-definite weighted Laplacian, so that $\Delta_{\vol}$ is the Laplace-Beltrami operator of $M$.
 We highlight that this shows that $T_\mu P^\infty(M)$ is in a one-to-one relation with the image of $\Delta_\mu$, so the Hodge decomposition theorem allows one to analyze projections of forms onto this space. To do so, we denote by $G_\mu$ the Green's operator for $d^*_\mu d$ on $L^2(M,\mu)$. That is, if $\int_M f\,d\mu=0$ and $\phi=G_\mu f$, then $\phi$ is the unique solution of $-\Delta_\mu\phi=f$ satisfying $\int_M\phi\,d\mu=0$. In particular $G_\mu 1=0$.

If $d\mu$ is a volume form, we define the \textbf{divergence} $\operatorname{div}_\mu(\xi)$ of $\xi \in T_\mu P(M)$ via
\begin{equation}
    \label{divmu}
    \nabla\cdot(\mu \xi)=\operatorname{div}_\mu(\xi)d\mu.
\end{equation}

If we denote the divergence with respect to the volume measure $\vol$ simply by $\operatorname{div}$, we have, for $d\mu=\rho\,d\vol$, 
$$\operatorname{div}_\mu(\xi)=\rho^{-1}\operatorname{div}(\rho\xi)=\operatorname{div}(\xi)+\rho^{-1}\langle\nabla\rho,\xi\rangle.$$
In particular,
$$\Delta_\mu\psi=\Delta_M\psi+\rho^{-1}\langle\nabla \rho,\nabla\psi \rangle;$$
for every $\psi\in C^\infty(M)$, in which $\Delta_M$ denotes the Laplacian of $M$.

%%%%%%%%%%%%%%%%%%%%%%%%%%%%%%%%%%%%%%%%%%%%%%%%%%%%%%%%%%%%%%%%%%%%%%%%%%%%%%%

\subsubsection*{On Tangent Spaces}

Here we further discuss the Hilbert space structure of the tangent spaces of $P(M)$.

Set
$$L^2_0(\mu):=\overline{\{\nabla\psi:\,\psi\in C^\infty(M)\}}^{L^2(\mu)}.$$

So that, if $\mu_t$ is a curve with derivative $-\nabla\cdot (Z\mu)$ at $\mu=\mu_0$, we may rewrite it in terms of the divergence operator with respect to $\mu$, via $$\dot{\mu}_0=-\operatorname{div}_\mu(Z)\mu.$$

So the tangent space of $P(M)$ at $\mu$ may be described as
$$T_\mu P(M):=\left\{-\operatorname{div}_\mu(Z)\mu:Z\in L^2_0(\mu)\right\}.$$

Integrating by parts gives $\int_M u\,\Delta_\mu u\,d\mu=-\int_M|\nabla u|^2d\mu$, so $\Delta_\mu u=0$ if and only if $du=0$. Hence $\ker(\Delta_\mu)$ consists of the constants, and its orthogonal complement in $L^2(\mu)$ is the set of functions with vanishing $\mu$-mean. Therefore $f\mu\in T_\mu P(M)$ if and only if $\int_M f\,d\mu=0$.

Furthermore, if we consider the distribution space $D(\mu):=\{Z\mu: Z\in L^2(\mu)\}$ then the map 
\begin{align*}
    V_\mu: & L ^2(\mu) \to  D(\mu)\\
    & Z \mapsto -\operatorname{div}_\mu(Z)\mu
\end{align*} is an isomorphism between $L^2_0(\mu)$ and $T_\mu P(M)$.

For simplicity, given $\phi\in C^\infty(M)$ we write
 $$V_\phi|_\mu:=V_{\mu}(\nabla \phi).$$
 When there is no risk of ambiguity, we simply denote $V_\phi|_\mu$ by $V_\phi$.

%%%%%%%%%%%%%%%%%%%%%%%%%%%%%%%%%%%%%%%%%%%%%%%%%%%%%%%%%%%%%%%%%%%%%%%%%%%%%%%%%%%%%%%%%%%%%%%%%%%%%%%%%%%%%%%%%%%%%%%%%%%%%%%%%%%%%%%%%%%%%%%%%%%%%%%%%%%%%%%%%%%%%%%%%%%%

\subsubsection*{Curves with constant velocity field}

We now single out a family of curves whose velocity field is represented by the same potential at every point. Given $\psi\in C^\infty(M)$, we call $V_\psi$ a \emph{constant} vector field, meaning that the same $\psi\in C^\infty(M)$ represents it at each tangent space $T_\mu P(M)$, regardless of $\mu$. These curves are not geodesics, but they are convenient for computing the action of a tangent vector on differentiable functions on $P^\infty(M)$.

\begin{prop}\label{prop:constant-field}
	Fix $\psi\in C^\infty(M)$ and denote by $\phi_t:M\to M$ the flow of $\nabla\psi$. Then, for any $\mu_0\in P(M)$, the curve $\mu_t:=(\phi_t)_*\mu_0$ has constant velocity field, that is, $\dot{\mu_t}=V_\psi|_{\mu_t}$ for every $t$.
\end{prop}

\begin{proof}
	Since $\mu_t=(\phi_t)_*\mu_0$ and $\phi_t$ is the flow of $\nabla\psi$,
	\begin{align*}
		\frac{d}{dt}\int_Mf(x)\mu_t(dx)&=\frac{d}{dt}\int_Mf(\phi_{t}(x))\mu_0(dx)=\int_M \langle\nabla f(\phi_t(x)),\nabla\psi(\phi_t(x))\rangle \mu_0(dx)\\
		&=\int_M \langle\nabla f(x), \nabla\psi(x)\rangle \mu_t(dx)=-\int_Mf\,\nabla\cdot(\mu_t \nabla\psi)(dx),
	\end{align*}
	so that $\dot{\mu_t}=V_{\psi}|_{\mu_t}$.
\end{proof}

This family admits a description in terms of the natural representation of $\operatorname{Diff}(M)$ on $C^0(M)$ induced by the cocycle of push-forward densities, a point of view we intend to explore elsewhere.

%%%%%%%%%%%%%%%%%%%%%%%%%%%%%%%%%%%%%%%%%%%%%%%%%%%%%%%%%%%%

\subsubsection*{Tangent Spaces and Smooth Functions on \texorpdfstring{$P^\infty(M)$}{P(M)}}

Here we analyse the differential structure of $P^\infty(M)$.

Following \cite{lott2006some}, we consider on the dense subset $P^\infty(M)$ introduced in Subsection \ref{secwasserstein} the space of fields
\begin{equation}\label{tangentspace}
	\widetilde{T}_\mu P^\infty(M)=\{V_\psi:\,\psi\in C^\infty(M)\},
\end{equation}
where $P^\infty(M)$ carries the smooth structure of \cite{kriegl1997convenient} rather than the one induced by the inclusion $P^\infty(M)\subset P(M)$, so that $\psi\mapsto V_\psi$ induces an isomorphism $C^{\infty}(M)/\mathbb{R}\to \widetilde{T}_\mu P^\infty(M)$. Observe that if $\operatorname{div}_\mu(X)=0$ and $T_t$ is the map $x\mapsto \exp_x(tX_x)$, then
$$\left.\frac{d}{dt}\right|_{t=0}(T_t)_*\mu=0,$$
so that the directions of the vector fields $X$ with $\operatorname{div}_\mu(X)=0$ do not move $\mu$ to first order, and only gradients need to be considered. We therefore define the \textbf{tangent space} of $P^\infty(M)$ at $\mu$ as
$$T_\mu P^\infty(M):=\widetilde{T}_\mu P^\infty(M)=\{V_\psi:\,\psi\in C^\infty(M)\}\cong C^\infty(M)/\mathbb{R}.$$

Following \cite{ding2021geometry}, we now fix a dense subset of $C^\infty(P^\infty(M))$ on which derivations can be computed explicitly. 

Given $\varphi\in C^\infty(M)$ we define the smooth function $F_\varphi\in C^\infty(P(M))$ via
$$F_\varphi(\mu)=\int_M \varphi(x)\mu(dx).$$
These functions separate the points of $P_2(M)$, since two measures agreeing against every $\varphi\in C^\infty(M)$ agree against every continuous function by density, and direct calculations show that 
$$V_\psi F_\varphi(\mu)=\langle V_\varphi,V_\psi\rangle_{\mu}=\int_M\langle\nabla\varphi(x),\nabla\psi(x)\rangle\mu(dx)=F_{\langle\nabla\varphi,\nabla\psi\rangle}(\mu).$$

A \emph{polynomial} function on $P(M)$ is a function of the form
$$F=F_{\varphi_1}\cdots F_{\varphi_k}$$
for $k\in\mathbb{N}$ and $\varphi_j\in C^\infty (M)$ for $1\leq j\leq k$. By the Stone--Weierstrass theorem, the space of polynomial functions is dense in the space of continuous functions on $P(M)$. For a more detailed reading we recommend \cite[Section 1.2]{ding2021geometry}.

%if $\mu_t$ is a curve with derivative $\dot{\mu}_t=X_t$, we have that 
%\begin{equation}
%\label{repdiff}
%\dot{\mu}_t=\frac{1}{2}U(X_t)^*\mu_t.
%\end{equation}

\subsubsection*{On Derivatives of Smooth Curves on \texorpdfstring{$P^{\infty}(M)$}{P(M)}}

Here we highlight an important basic relation on derivatives of smooth curves on $P^\infty(M)$.

If $\mu_t=\rho_t\,d\vol$ is a path on $P^\infty(M)$ whose derivative field is given by $\nabla\psi_t\in T_{\mu_t}P^\infty(M)$, then $$\dot{\mu_t}=-\nabla\cdot(\mu_t\nabla\psi_t)=-\operatorname{div}_{\mu_t}(\nabla\psi_t)\mu_t=-\operatorname{div}_{\rho_t d\vol }(\nabla\psi_t)\rho_t d\vol =-\operatorname{div}(\rho_t\nabla\psi_t)d\vol.$$

So, on the one hand we have
$$\frac{d}{dt}\int_M f \mu_t= -\int_M f\Delta_{\mu_t}\psi_td\mu_t= -\int_M f\operatorname{div}(\rho_t\nabla\psi_t)d\vol.$$
On the other,
$$\frac{d}{dt}\int_M f \mu_t= \frac{d}{dt}\int_M f \rho_t d\vol =\int_M f \dot{\rho}_t d\vol.$$
Therefore, almost everywhere,
\begin{equation}\label{derivativelaplace}
	\dot{\rho_t}d\vol =-\operatorname{div}(\rho_t\nabla\psi_t)d\vol =-\Delta_{\mu_t}\psi_t d\mu_t.
\end{equation}

%%%%%%%%%%%%%%%%%%%%%%%%%%%%%%%%%%%%%%%%%%%%%%%%%%%%%%%%%%%%%%%%%%%%%%%%%%%%%%%%%%%%%%%%%%%%%%%%%%%%%%%%%%%%%%%%%%%%%%%%%%%%%%%%%%%%%%%%%%%%%%%%%%%%%%%%%%%%%%%%%%%%%%%%%%%%%%%%%%%%%%%%%%%%%%%%%%%%%%%%%%%%%%%%%%%%%%%%%%%%%%%%%%%%%%%%%%%%%%%%%%%%%%%%%%%%%%%%%%%%%%%%%%%%%%%%%%%%%%%%%%%%%%%%%%%%%%%%%%%%
%%%%%%%%%%%%%%%%%%%%%%%%%%%%%%%%%%%%%%%%%%%%%%%%%%%%%%%%%%%%%%%%%%%%%%%%%%%%%%%%%%%%%%%%%%%%%%%%%%%%%%%%%%%%%%%%%%%%%%%%%%%%%%%%%%%%%%%%%%%%%%%%%%%%%%%%%%%%%%%%%%%%%%%%%%%%%%%%%%%%%%%%%%%%%%%%%%%%%%%%%%%%%%%%%%%%%%%%%%%%%%%%%%%%%%%%%%%%%%%%%%%%%%%%%%

\subsection{Riemannian geometry}\label{sec2}

Here we present Otto's Riemannian metric on $P^{ac}(M)$, the Lie bracket of constant vector fields and, via the Koszul formula, the Levi-Civita connection. These objects were computed by Lott on $P^\infty(M)$ in local coordinates \cite{lott2006some} and by Ding and Fang on $P_2(M)$ \cite{ding2021geometry}. The expressions below record them in the coordinate-free notation of the previous subsection, which is the form used in the sequel.

\subsubsection*{Riemannian Metric}

Otto's Riemannian metric on $P(M)$ at a point $\mu=\rho\,d\vol$ is given, for $V_\mu(X),V_\mu(Y)\in T_\mu P(M)$, by
\begin{equation}\label{ottoprod}
	\langle V_\mu(X),V_\mu(Y)\rangle_\mu=\int_M\langle X,Y\rangle\, d\mu.
\end{equation}
By construction, $V_\mu$ is then an isometry from $L^2_0(\mu)$ onto $T_\mu P(M)$. For $X=\nabla\phi_1$ and $Y=\nabla\phi_2$ with $\phi_1,\phi_2\in C^\infty(M)$, integration by parts gives
\begin{equation}\label{ottopotential}
	\langle V_{\phi_1},V_{\phi_2}\rangle_\mu=\int_M\langle \nabla\phi_1,\nabla\phi_2\rangle\, d\mu=-\int_M \phi_1\,\nabla\cdot(\mu \nabla\phi_2)=-\int_M \phi_1\,\Delta_\mu\phi_2\, d\mu.
\end{equation}

 In local coordinates, using the musical isomorphisms $\flat:TM\to T^*M$ to define $\nabla^i=g^{ij}\nabla_i$ (Einstein's notation is being used) with $g^{ij}$ being the matrix coefficients of the Riemannian metric of $M$ and $\nabla_i$ being the covariant derivative in the $i$-th coordinate direction, we have that, by \cite[Equation 2.6]{lott2006some},
 $$\Delta_\mu\psi\,d\mu=\nabla^i(\rho\nabla_i\psi)\,d\vol,$$ 
and thus
$$\nabla\cdot(\mu\nabla\phi_2)=\nabla^i(\rho\nabla_i\phi_2)\,d\vol.$$
Therefore, by (\ref{ottopotential}),
$$\langle V_{\phi_1},V_{\phi_2}\rangle_\mu=-\int_M\phi_1\,\nabla^i(\rho\nabla_i\phi_2)\,d\vol.$$

\subsubsection*{Lie Bracket and Levi-Civita Connection}

\begin{prop}\label{prop:lie-bracket}
    Given $\phi_1,\phi_2\in C^\infty(M)$, the Lie bracket of constant vector fields is given by
    $$[V_{\phi_1},V_{\phi_2}]=V_{\theta_\mu}$$
    with
    $$\theta_\mu=\Delta_\mu^{-1}\operatorname{div}_\mu(\nabla\phi_1\Delta_\mu\phi_2-\nabla\phi_2\Delta_\mu\phi_1)=G_\mu d_\mu^*\flat(\nabla^2\phi_2(\nabla\phi_1)-\nabla^2\phi_1(\nabla\phi_2)).$$
\end{prop}

    \begin{proof}
  	Let $\phi^i_t$ denote the flow of $\nabla\phi_i$, so that by Proposition \ref{prop:constant-field} the curve $(\phi^i_t)_*\mu$ has constant velocity field $V_{\phi_i}$. Then
  	\begin{align*}
  		[V_{\phi_1},V_{\phi_2}]F_f(\mu)&=(V_{\phi_1}(V_{\phi_2}))F_f(\mu)-(V_{\phi_2}(V_{\phi_1}))F_f(\mu)\\
       &=\left.\frac{d}{dt}\right|_{t=0}(V_{\phi_2}F_f)((\phi^1_t)_*\mu)-\left.\frac{d}{dt}\right|_{t=0}(V_{\phi_1}F_f)((\phi^2_t)_*\mu)\\
       & =\left.\frac{d}{dt}\right|_{t=0}\int_M\langle\nabla f(x),\nabla\phi_2(x)\rangle ((\phi^1_t)_*\mu)(dx)\\
       & -\left.\frac{d}{dt}\right|_{t=0}\int_M\langle\nabla f(x),\nabla\phi_1(x)\rangle ((\phi^2_t)_*\mu)(dx)\\
       & =-\int_M\langle \nabla f,\nabla\phi_2\rangle\,\Delta_\mu \phi_1\, d\mu+\int_M\langle \nabla f,\nabla\phi_1\rangle\,\Delta_\mu \phi_2\, d\mu\\
	   &=\int_M\langle\nabla f,\nabla\phi_1\Delta_\mu\phi_2-\nabla\phi_2\Delta_\mu\phi_1\rangle\, d\mu.
\end{align*}
        By \cite[Proposition 3.2]{ding2021geometry}, the projection of $\mathcal{C}_{\phi_1,\phi_2}(\mu)=\nabla\phi_1\Delta_\mu\phi_2-\nabla\phi_2\Delta_\mu\phi_1$ onto $T_\mu P^\infty(M)$ is $V_{\theta_\mu}$ with
    $$\theta_\mu=\Delta_\mu^{-1}\operatorname{div}_\mu(\nabla\phi_1\Delta_\mu\phi_2-\nabla\phi_2\Delta_\mu\phi_1).$$
Thus,
    $$[V_{\phi_1},V_{\phi_2}]F_f(\mu)=V_{\theta_\mu}F_f(\mu).$$
    This formula can be extended to polynomials and by continuity to smooth functions so that we have the identification
    $$[V_{\phi_1},V_{\phi_2}]=V_{\theta_\mu}.$$

    Note that $\Delta_\mu^{-1}\operatorname{div}_\mu=G_\mu d^*_\mu\flat$, since $\Delta_\mu=-d^*_\mu d$ and $d^*_\mu\flat(\xi)=-\operatorname{div}_\mu(\xi)$, so that the two signs cancel. The identity
\begin{equation}\label{eq:bracket-potential}
	\theta_\mu=G_\mu d_\mu^*\flat(\nabla^2\phi_2(\nabla\phi_1)-\nabla^2\phi_1(\nabla\phi_2))
\end{equation}
is then \cite[Equation 5.11]{lott2006some}.
\end{proof}

This way,
\begin{equation}\label{bracket}
	\langle V_{\phi_1},[V_{\phi_2},V_{\phi_3}] \rangle_\mu=\int_M\langle\nabla \phi_2\,\Delta_\mu \phi_3-\nabla \phi_3\,\Delta_\mu\phi_2,\nabla\phi_1\rangle\, d\mu.
\end{equation}
On the other hand,
\begin{equation}\label{product}
	V_{\phi_1}\langle V_{\phi_2},V_{\phi_3}\rangle_\mu=\int_M\langle \nabla\phi_1,\nabla\langle\nabla \phi_2,\nabla\phi_3\rangle\rangle\, d\mu= -\int_M\langle \nabla\phi_2,\nabla\phi_3\rangle\,\Delta_\mu \phi_1\, d\mu.
\end{equation}

Let us denote the Levi-Civita connection of $P^{\infty}(M)$ by $\overline{\nabla}_{(\cdot)}(\cdot)$. By combining the equations (\ref{bracket}) and (\ref{product}) with the Koszul formula:
\begin{align*}
    2\langle\overline{\nabla}_{V_{\phi_1}}V_{\phi_2},V_{\phi_3}\rangle_\mu= 
    & V_{\phi_1}\langle V_{\phi_2},V_{\phi_3}\rangle_\mu+ V_{\phi_2}\langle V_{\phi_3},V_{\phi_1}\rangle_\mu- V_{\phi_3}\langle V_{\phi_1},V_{\phi_2}\rangle_\mu\\
    &+\langle V_{\phi_3},[V_{\phi_1},V_{\phi_2}] \rangle_\mu -\langle V_{\phi_2},[V_{\phi_1},V_{\phi_3}] \rangle_\mu -\langle V_{\phi_1},[V_{\phi_2},V_{\phi_3}] \rangle_\mu,
\end{align*}
one gets that, at $\mu=\rho\,d\vol$,
\begin{equation}\label{eqding}
    \langle\overline{\nabla}_{V_{\phi_1}}V_{\phi_2},V_{\phi_3}\rangle_\mu=\frac{1}{2}\left(\int_M\langle\nabla\langle\nabla\phi_1,\nabla\phi_2\rangle,\nabla\phi_3\rangle d\mu+\int_M\langle\Delta_\mu\phi_2\nabla\phi_1-\Delta_\mu\phi_1\nabla\phi_2,\nabla\phi_3 \rangle d\mu\right).
\end{equation}
Equivalently, by \cite[Lemma 3]{lott2006some},
\begin{equation}\label{prodconnect}
	\langle\overline{\nabla}_{V_{\phi_1}}V_{\phi_2},V_{\phi_3}\rangle_\mu=\int_M \langle\nabla \phi_1,\nabla^2\phi_2(\nabla\phi_3) \rangle\, d\mu,\end{equation}
which in local coordinates reads $\int_M\nabla_i\phi_1\,\nabla_j\phi_3\,\nabla^i\nabla^j\phi_2\, \rho\, d\vol$.

\begin{lemma}\label{Lemma4}
    At $\mu=\rho\,d\vol$, we have $\overline{\nabla}_{V_{\phi_1}}V_{\phi_2}= V_\phi$, where
$$\phi= G_\mu d_\mu^*\flat(\nabla^2\phi_2(\nabla \phi_1)).$$
\end{lemma}

    \begin{proof}
    	The proof can be found in \cite[Lemma 4]{lott2006some} in local coordinates. We present it here adapted to our intrinsic language. Let $\Pi_\mu$ denote the orthogonal projection onto $\overline{\operatorname{Im}(d)}\subset \Omega^1_{L^2}(M,\mu)$, so that $dG_\mu d^*_\mu=\Pi_\mu$ by the definition of $G_\mu$. Given $\phi_3\in C^\infty(M)$, and using that $d\phi_3\in\operatorname{Im}(d)$ to drop $\Pi_\mu$ in the third equality,
    \begin{align*}
        \langle V_{\phi_3},V_\phi\rangle_\mu & =\int_M\langle d\phi_3,dG_\mu d_\mu^*\flat(\nabla^2\phi_2(\nabla \phi_1))\rangle\, d\mu\\
        & =\int_M\langle d\phi_3,\Pi_\mu\flat(\nabla^2\phi_2(\nabla \phi_1))\rangle\, d\mu= \int_M\langle d\phi_3,\flat(\nabla^2\phi_2(\nabla \phi_1))\rangle\, d\mu\\
    	&=\int_M\langle \nabla\phi_3,\nabla^2\phi_2(\nabla \phi_1)\rangle\, d\mu=\int_M\langle \nabla\phi_1,\nabla^2\phi_2(\nabla \phi_3)\rangle\, d\mu=\langle V_{\phi_3},\overline{\nabla}_{V_{\phi_1}}V_{\phi_2}\rangle_\mu,
	\end{align*}
    where the second to last equality uses the symmetry of $\nabla^2\phi_2$ and the last one is (\ref{prodconnect}). Since this holds for every $\phi_3\in C^\infty(M)$ and the metric is non-degenerate on $T_\mu P^\infty(M)$, we conclude that $V_\phi=\overline{\nabla}_{V_{\phi_1}}V_{\phi_2}$.
\end{proof}

Another useful relation that arises from (\ref{eqding}) is given by \cite[Theorem 3.3]{ding2021geometry}:
\begin{equation}\label{theo3.3}
	\overline{\nabla}_{V_{\phi_1}}V_{\phi_2}=\frac{1}{2}V_{\langle\nabla\phi_1,\nabla\phi_2\rangle}+\frac{1}{2}[V_{\phi_1},V_{\phi_2}].
\end{equation}
It is this last identity, together with (\ref{prodconnect}), that the parallelization of the next section turns into explicit Christoffel symbols.

%%%%%%%%%%%%%%%%%%%%%%%%%%%%%%%%%%%%%%%%%%%%%%%%%%%%%%%%%%%%%%%%%%%%%%%%%%%%%%%%%%%%%%%%%%%%%%%%%%%%%%%%%%%%%%%%%%%%%%%%%%%%%%%%%%%%%%%%%%%%%%%%%%%%%%%%%%%%%%%%%%%%%%%%%%%%%%%%%%%%%%%%%%%%%%%%%%%%%%%%%%%%%%%%%%%%%%%%%%%%%%%%%%%%%%%%%%%%%%%%%%%%%%%%%%%%%%%%%%%%%%%%%%%%%%%%%%%%%%%%%%%%%%%%%%%%%%%%%%%%%%%%%%%%%%%%%%%%%%%%%%%%%%%%%%%%%%%%%%%%%%%%%%%%%%%%%%%%%%%%%%%%%%%%%%%%%%%%%%%%%%%%%%%%%%%%%%%%%%%%

\section{A Parallelization of \texorpdfstring{$P^\infty(M)$}{P(M)}}\label{secparallelization}

In this section we restrict the previous constructions to the subspace $P^\infty(M)$ and make the Riemannian structure more explicit. Our goal is to construct a global parallelization of the tangent bundle in terms of potentials.

To this end, we fix a basis of eigenfunctions of the Laplace-Beltrami operator of $M$ and consider the associated family of vector fields on $P^\infty(M)$. This allows us to represent tangent vectors by sequences of coefficients and to express the metric and the Levi-Civita connection in this basis.

In particular, we introduce the matrix coefficients of the metric and the corresponding Christoffel symbols, and we write the equations of parallel transport and of geodesics in the frame. This yields a coordinate description of the geometry of $P^\infty(M)$ compatible with the intrinsic formalism developed in the previous section and with Lott's extrinsic formalism in \cite{lott2006some}.

The existence of such a frame is stated by Lott \cite[Section 4]{lott2006some}, without proof and without specifying the topology in which the family is a basis. Theorem \ref{thm:parallelization} gives a proof for bases of eigenfunctions, with that topology made explicit. The equation of parallel transport and the curvature formulas are those of Lott \cite[Sections 4 and 5]{lott2006some}, written in the intrinsic notation of Section \ref{sec:preliminaries}. What the frame adds are the coefficients of the metric and of the connection, which Section \ref{secExample} computes in closed form on the circle.

\subsection{Global Basis}

The whole construction rests on the following observation: a basis of eigenfunctions of the Laplacian of $M$ induces, simultaneously at every measure, a basis of the tangent space. Recall from Subsection \ref{sec1} that $V_\mu$ identifies the space $L^2_0(\mu)$, the closure in $L^2(\mu)$ of the gradients of smooth functions, with $T_\mu P(M)$, and that $T_\mu P^\infty(M)$ is the dense subspace of the $V_\psi$ with $\psi\in C^\infty(M)$.

\begin{theo}\label{thm:parallelization}
	Let $\{\phi_i\}_{i\geq1}\subset C^\infty(M)$ be eigenfunctions of the Laplace-Beltrami operator, $\Delta\phi_i=-\lambda_i\phi_i$ with $\lambda_i>0$, such that $\{1\}\cup\{\phi_i\}_{i\geq1}$ is an orthogonal basis of $L^2(\vol)$. Then, for every $\mu\in P^\infty(M)$, the following hold.
	\begin{enumerate}
		\item Every $X\in L^2_0(\mu)$ has an expansion $X=\sum_ic_i\nabla\phi_i$ converging in $L^2(\mu)$, and the coefficients are unique: if $\sum_ic_i\nabla\phi_i$ converges to $0$ in $L^2(\mu)$, then $c_i=0$ for every $i$. Equivalently, $\{V_{\phi_i}\}_{i\geq1}$ is a Schauder basis of the Hilbert space $T_\mu P(M)$ with Otto's inner product (\ref{ottoprod}).
		\item If $\psi\in C^\infty(M)$, then $V_\psi=\sum_ic_iV_{\phi_i}$ with $c_i=\langle\psi,\phi_i\rangle_{L^2(\vol)}/\norm{\phi_i}^2_{L^2(\vol)}$, and $\sum_ic_i\phi_i$ converges to $\psi-\int_M\psi\,d\vol$ in every $C^k(M)$.
	\end{enumerate}
	Since the family does not depend on $\mu$, $\{V_{\phi_i}\}_{i\geq1}$ is a global frame for $TP^\infty(M)$.
\end{theo}

\begin{proof}
	Write $\mu=\rho\,\vol$ with $\rho\in C^\infty(M)$ positive. Since $M$ is compact and $\rho$ is continuous and positive, $m:=\inf_M\rho>0$ and $K:=\sup_M\rho<\infty$, so that
	$$m\norm{\xi}^2_{L^2(\vol)}\leq\norm{\xi}^2_{L^2(\mu)}\leq K\norm{\xi}^2_{L^2(\vol)}$$
	for every $\xi\in L^2(\vol)$. The norms $L^2(\vol)$ and $L^2(\mu)$ are therefore equivalent, both on functions and on vector fields, and in particular $L^2_0(\mu)=L^2_0(\vol)$ as sets.

	We prove (2) first. For $\psi\in C^\infty(M)$ and $k\geq0$, integrating by parts $k$ times gives $\langle\psi,\phi_i\rangle_{L^2(\vol)}=(-\lambda_i)^{-k}\langle\Delta^k\psi,\phi_i\rangle_{L^2(\vol)}$, so the coefficients $c_i$ decay faster than any power of $\lambda_i$. By Weyl's law and the bounds $\norm{\phi_i}_{C^k}\leq C_k(1+\lambda_i)^{N_k}\norm{\phi_i}_{L^2(\vol)}$, which follow from elliptic regularity and the Sobolev embedding, the series $\sum_ic_i\phi_i$ converges in every $C^k(M)$. Its limit is $\psi-\int_M\psi\,d\vol$, because this is its limit in $L^2(\vol)$. In particular $\nabla\psi=\sum_ic_i\nabla\phi_i$ uniformly on $M$, hence in $L^2(\mu)$, and $V_\psi=\sum_ic_iV_{\phi_i}$ since $V_\mu$ is an isometry by (\ref{ottoprod}).

	For (1), integration by parts gives $\langle\nabla\phi_i,\nabla\phi_j\rangle_{L^2(\vol)}=\lambda_j\langle\phi_i,\phi_j\rangle_{L^2(\vol)}$, so the gradients $\nabla\phi_i$ are orthogonal in $L^2(\vol)$, with $\norm{\nabla\phi_i}^2_{L^2(\vol)}=\lambda_i\norm{\phi_i}^2_{L^2(\vol)}$. By (2) their closed span in $L^2(\vol)$ contains every gradient of a smooth function, so it is $L^2_0(\vol)$, and the normalized gradients form an orthonormal basis of $L^2_0(\vol)$. Every $X\in L^2_0(\mu)=L^2_0(\vol)$ is then the sum of its expansion in $L^2(\vol)$, hence in $L^2(\mu)$ by the equivalence of norms. If $\sum_ic_i\nabla\phi_i$ converges to $0$ in $L^2(\mu)$, it converges to $0$ in $L^2(\vol)$, and taking the inner product with $\nabla\phi_j$ gives $c_j\lambda_j\norm{\phi_j}^2_{L^2(\vol)}=0$, so $c_j=0$. The statement about $T_\mu P(M)$ follows because $V_\mu$ is an isometry.
\end{proof}

Theorem \ref{thm:parallelization} is a proof, for bases of eigenfunctions, of the observation of Lott recalled above, and some care in the choice of basis is needed. The constant function has to be excluded, since $V_1=0$. For a general orthonormal basis of $L^2(\vol)$ whose expansions converge in $C^1$, the argument above still gives expansions of the $V_\psi$ with $\psi$ smooth, but not the uniqueness of the coefficients, because the gradients need not be orthogonal. On a compact Lie group with a bi-invariant metric the Laplacian acts on the span of the matrix coefficients of each irreducible representation by a scalar, so by the Peter-Weyl theorem an orthogonal basis of real eigenfunctions can be extracted from the real and imaginary parts of the matrix coefficients. On the circle this is the Fourier basis used in Section \ref{secExample}.

We stress that $\{V_{\phi_i}\}$ is orthogonal at $\mu=\vol$ but not at other measures: the inner products $\langle V_{\phi_i},V_{\phi_j}\rangle_\mu$ vary with $\mu$, and it is precisely this dependence that gives content to the Christoffel symbols computed below.

\subsection{Metric Coefficients, Christoffel Symbols and Geodesics}

The global frame given by Theorem \ref{thm:parallelization} allows us to define matrix coefficients for the Riemannian metric of $P^\infty(M)$ and the associated Christoffel symbols, and thereby to write the geodesic equation as a system of ordinary differential equations. 

We start by defining the matrix coefficients of the Riemannian metric $\overline{g}$ of $P^\infty(M)$ by
\begin{equation}
	\label{matrixmetric}\overline{g}_{ij}(\mu)=\langle V_{\phi_i},V_{\phi_j}\rangle_\mu=\int_M\langle\nabla\phi_i,\nabla\phi_j\rangle\,d\mu,
\end{equation}
and its associated Christoffel symbols, $\Gamma^k_{ij}$ via
\begin{equation}\label{Christoffel}
    \overline{\nabla}_{V_{\phi_i}}V_{\phi_j}=\sum_k \Gamma^k_{ij} V_{\phi_k}.
\end{equation}
By Lemma \ref{Lemma4}, $\overline{\nabla}_{V_{\phi_i}}V_{\phi_j}=V_{\chi_{ij}}$ with $\chi_{ij}=G_\mu d^*_\mu\flat(\nabla^2\phi_j(\nabla\phi_i))$, a smooth function by elliptic regularity. Theorem \ref{thm:parallelization} therefore shows that the $\Gamma^k_{ij}$ are well defined, and that
\begin{equation}\label{eq:christoffel-explicit}
	\Gamma^k_{ij}(\mu)=\frac{\big\langle G_\mu d^*_\mu\flat\big(\nabla^2\phi_j(\nabla\phi_i)\big),\phi_k\big\rangle_{L^2(\vol)}}{\norm{\phi_k}^2_{L^2(\vol)}},
\end{equation}
an expression which does not require inverting the matrix $\overline{g}(\mu)$. The symmetric part of the Christoffel symbols does not depend on $\mu$. Indeed, (\ref{theo3.3}) gives $\overline{\nabla}_{V_{\phi_i}}V_{\phi_j}+\overline{\nabla}_{V_{\phi_j}}V_{\phi_i}=V_{\langle\nabla\phi_i,\nabla\phi_j\rangle}$, so that
\begin{equation}\label{eq:christoffel-sym}
	\Gamma^k_{ij}+\Gamma^k_{ji}=\frac{\big\langle\langle\nabla\phi_i,\nabla\phi_j\rangle,\phi_k\big\rangle_{L^2(\vol)}}{\norm{\phi_k}^2_{L^2(\vol)}}.
\end{equation}
The dependence on $\mu$ is carried by the antisymmetric part, which is given by the Lie bracket $[V_{\phi_i},V_{\phi_j}]$ of Proposition \ref{prop:lie-bracket}.

Now, fix $V_X=\sum_{i} x_i V_{\phi_i}$ and $V_Y=\sum_{i} y_i V_{\phi_i}$ two smooth vector fields of $P^{\infty}(M)$ with $x_i,y_i\in C^\infty( P^{\infty}(M))$ for every $i\in\mathbb{N}$. Then,
\begin{align*}
    [V_X,V_Y]&= \sum_j [V_X,y_j V_{\phi_j}]=\sum_jV_X(y_j)V_{\phi_j}+y_j[V_X,V_{\phi_j}]\\
    &=\sum_{i,j} x_i V_{\phi_i}(y_j)V_{\phi_j}-y_j V_{\phi_j}(x_i)V_{\phi_i}+x_iy_j[V_{\phi_i},V_{\phi_j}].
\end{align*}

So, by Equation (\ref{theo3.3}),
\begin{align*}
   \overline{\nabla}_{V_X}V_Y & =\sum_{i,j}x_i\left(y_j\left(\overline{\nabla}_{V_{\phi_i}}V_{\phi_j}\right)+V_{\phi_i}(y_j)V_{\phi_j}\right)\\
    &=\sum_{i,j}\left(x_iy_j\left(\frac{1}{2}V_{\langle \nabla{\phi_i},\nabla{\phi_j}\rangle}+\frac{1}{2}[V_{\phi_i},V_{\phi_j}]\right)+x_iV_{\phi_i}(y_j)V_{\phi_j}\right)\\
    &=\frac{1}{2}\sum_{i,j}x_iy_j\left(V_{\langle \nabla{\phi_i},\nabla{\phi_j}\rangle}+[V_{\phi_i},V_{\phi_j}]\right)+\sum_j V_X(y_j)V_{\phi_j}. 
    \end{align*}

We finally have the necessary framework to study the notion of parallelism of a vector field along a curve and hence to study the geodesics of our spaces.

%%%%%%%%%%%%%%%%%%%%%%%%%%%%%%%%%%%%%%%%%%%%%

    \begin{prop}\label{prop:parallel}
	Let $\mu_t=\rho_t\,d\vol$ be a smooth curve in $P^\infty(M)$ with $\dot{\mu}_t=V_{\psi_t}$, and let $\eta_t\in C^\infty(M)$ depend smoothly on $t$. The vector field $V_{\eta_t}$ along $\mu_t$ is parallel if and only if one of the following equivalent conditions holds for every $t$:
        \begin{enumerate}
            \item $\dot{\eta}_t+G_{\mu_t}d^*_{\mu_t}\flat\big(\nabla^2\eta_t(\nabla\psi_t)\big)$ is constant on $M$,
            \item $2\dot{\eta}_t+\langle\nabla\eta_t,\nabla\psi_t\rangle+\theta_t$ is constant on $M$,
        \end{enumerate}
    where $\theta_t=\Delta_{\mu_t}^{-1}\operatorname{div}_{\mu_t}(\nabla \psi_t\Delta_{\mu_t}\eta_t-\nabla\eta_t\Delta_{\mu_t} \psi_t)$ is the potential of $[V_{\psi_t},V_{\eta_t}]$ given by Proposition \ref{prop:lie-bracket}.
    \end{prop}

    \begin{proof}
	For fixed $s$, $V_{\eta_s}$ is a vector field on $P^\infty(M)$, and the field along the curve is $t\mapsto V_{\eta_t}|_{\mu_t}$. By the Leibniz rule for the covariant derivative along a curve, $D_tV_{\eta_t}=\partial_sV_{\eta_s}|_{s=t}+\overline{\nabla}_{\dot{\mu}_t}V_{\eta_t}$, where in the second term $\eta_t$ is frozen. Since $\eta\mapsto V_\eta$ is linear, the first term is $V_{\dot{\eta}_t}$. Lemma \ref{Lemma4} with $\phi_1=\psi_t$ and $\phi_2=\eta_t$ then gives
	$$D_tV_{\eta_t}=V_{\dot{\eta}_t+G_{\mu_t}d^*_{\mu_t}\flat(\nabla^2\eta_t(\nabla\psi_t))}.$$
	Since $\rho_t>0$ and $M$ is connected, $V_\chi=0$ exactly when $\chi$ is constant, and this gives (1). By (\ref{theo3.3}) and Proposition \ref{prop:lie-bracket}, $\overline{\nabla}_{V_{\psi_t}}V_{\eta_t}=\frac12V_{\langle\nabla\psi_t,\nabla\eta_t\rangle}+\frac12V_{\theta_t}$, and this gives (2).
    \end{proof}

Condition (1) is the equation of parallel transport of Lott \cite[Section 4]{lott2006some}, written in intrinsic form. If we take $\eta_t=\psi_t$, the bracket vanishes and (2) becomes $2\dot{\psi}_t+|\nabla\psi_t|^2=\mathrm{const}$. A potential is defined only up to an additive constant, which may depend on $t$, so this constant can be absorbed into $\psi_t$, and we recover \cite[Proposition 4]{lott2006some} on the geodesics of $P^\infty(M)$.
    
    \begin{corollary}\label{corollarygeodesic}
    	The curve $\mu_t$ is a geodesic if and only if, after normalizing the additive constants of the $\psi_t$,
    	\begin{equation}\label{edogeodesica}
    		\frac{\partial\psi_t}{\partial t}+\frac{1}{2}|\nabla\psi_t|^2=0.
    	\end{equation}
    \end{corollary}
    
    It is a classical fact in optimal transport theory that (\ref{edogeodesica}) holds for the velocity fields of displacement interpolations, see \cite[Equation 13.16]{villani2009optimal}. In general its solutions are smooth only up to a finite time, at which the characteristics $x\mapsto\exp_x(t\nabla\psi_0(x))$ start to cross, as we make explicit on the circle in Proposition \ref{prop:geodesic-circle}.

In the frame of Theorem \ref{thm:parallelization}, write $\psi_t=\sum_k\psi_k(t)\phi_k$. Only the symmetric part of the Christoffel symbols enters the geodesic equation, and by (\ref{eq:christoffel-sym}) it does not depend on the measure.

    \begin{corollary}\label{cor:geochristoffel}
        A smooth curve $\mu_t$ in $P^\infty(M)$ whose velocity field is $V_{\psi_t}$ is a geodesic if and only if, for every $k\geq1$,
        \begin{equation}\label{geochristoffel}
        \psi'_k+\sum_{i,j}\psi_i\psi_j\Gamma_{ij}^k(\mu_t)=0.
        \end{equation}
	Equivalently,
	$$\psi'_k=-\frac{1}{2\norm{\phi_k}^2_{L^2(\vol)}}\sum_{i,j}\psi_i\psi_j\big\langle\langle\nabla\phi_i,\nabla\phi_j\rangle,\phi_k\big\rangle_{L^2(\vol)},\qquad k\geq1,$$
	which is a closed quadratic system for the coefficients $\psi_k$ alone.
    \end{corollary}

    \begin{proof}
	By the proof of Proposition \ref{prop:parallel} with $\eta_t=\psi_t$, the curve is a geodesic if and only if $V_{\dot{\psi}_t}+\overline{\nabla}_{V_{\psi_t}}V_{\psi_t}=0$, with $\psi_t$ frozen in the second term. Expanding $\psi_t$ in the frame, $\overline{\nabla}_{V_{\psi_t}}V_{\psi_t}=\sum_k\big(\sum_{i,j}\psi_i\psi_j\Gamma^k_{ij}\big)V_{\phi_k}$, where the double series converges absolutely because the coefficients of a smooth potential decay faster than any power of $\lambda_i$, while the $\Gamma^k_{ij}$ grow at most polynomially. The uniqueness in Theorem \ref{thm:parallelization} gives (\ref{geochristoffel}), and the second form follows from (\ref{eq:christoffel-sym}), since $\psi_i\psi_j$ is symmetric in $i$ and $j$.
    \end{proof}

The second form of the system is the Hamilton-Jacobi equation (\ref{edogeodesica}) written in the coefficients of the frame, since its right hand side is the $k$-th coefficient of $-\frac12|\nabla\psi_t|^2$. It is therefore not a new description of the geodesics, and it inherits the limitations recorded after Corollary \ref{corollarygeodesic}. The measure enters the Christoffel symbols only through their antisymmetric part, which is invisible to geodesics but not to parallel transport, as Proposition \ref{prop:parallel} shows.

%%%%%%%%%%%%%%%%%%%%%%%%%%%%%%%%%%%%%%%%%%%%%%%%%%%%%%%%%%%%%%%%%%%%%%%%%%%%%%%%%%%%%%%%%%%%%%%%%%%%%%%%%%%%%%%%%%%%%%%%%%%%%%%%%%%%%%%%%%%%%%%%%%%%%%

\subsection{Curvature}

In \cite[Section 5]{lott2006some} Lott computes the Ricci and sectional curvatures of $P^\infty(M)$ in terms of the vector fields $V_\phi$. Here we present these results in our coordinate-free setting.

Given $\phi,\psi\in C^\infty(M)$, define $T_{\phi\psi}\in\Omega^1_{L^2}(M,\mu)$ by
\begin{equation}\label{defT}
    T_{\phi\psi}=(I-\Pi_\mu)(\flat(\nabla^2\psi(\nabla\phi))).
\end{equation}
Remember that  $\Pi_\mu$ is the orthogonal projection onto $\overline{\operatorname{Im}(d)}\subset \Omega^1_{L^2}(M,\mu)$.

The tensor $T$ has a simple meaning. A one-form $\alpha\in\Omega^1_{L^2}(M,\mu)$ is orthogonal to $\operatorname{Im}(d)$ exactly when $\int_M\langle d\varphi,\alpha\rangle\,d\mu=0$ for every $\varphi\in C^\infty(M)$, that is, when the vector field $Z$ with $\flat(Z)=\alpha$ satisfies $\nabla\cdot(\mu Z)=0$ in the sense of distributions. Reading $Z$ as a velocity field, the continuity equation (\ref{conservation}) says that the corresponding flow of mass leaves $\mu$ unchanged to first order. Thus $\operatorname{Im}(d)^\perp$ consists of the infinitesimal motions of $M$ that do not move $\mu$. These are the directions quotiented out in the definition of $T_\mu P^\infty(M)$, and only the component $\Pi_\mu\alpha$ of a one-form is seen by the geometry of $P^\infty(M)$. The tensor $T_{\phi\psi}$ is precisely the invisible component of $\flat(\nabla^2\psi(\nabla\phi))$, and in the formula for the sectional curvature below it enters through the term $3|T_{\phi_1\phi_2}|^2_\mu$, which plays the role of the O'Neill term $3|A_XY|^2$ of a Riemannian submersion.

\begin{lemma}\label{lem:T-antisym}
	$T_{\phi\psi}+T_{\psi\phi}=0$.
\end{lemma}

\begin{proof}
    It follows from the facts that $I-\Pi_\mu$ projects away from $\operatorname{Im}(d)$ and that
    $$\flat(\nabla^2\psi(\nabla\phi)+\nabla^2\phi(\nabla\psi))=d\langle\nabla\phi,\nabla\psi\rangle.$$
\end{proof}

\begin{theo}\label{riemanniancurvature}
    Given $\phi_1,\phi_2,\phi_3,\phi_4\in C^\infty(M)$, the Riemannian curvature operator $\overline{R}$ of $P^\infty(M)$ is given by 
    \begin{align*}
        \langle \overline{R}(V_{\phi_1},V_{\phi_2})V_{\phi_3},V_{\phi_4}\rangle_\mu & = \int_M \langle R(\nabla\phi_1,\nabla\phi_2)\nabla\phi_3,\nabla\phi_4\rangle d\mu -2\langle T_{\phi_1\phi_2},T_{\phi_3\phi_4}\rangle_\mu\\
    &+\langle T_{\phi_2\phi_3},T_{\phi_1\phi_4}\rangle_\mu-\langle T_{\phi_1\phi_3},T_{\phi_2\phi_4}\rangle_\mu.
    \end{align*}
          
\end{theo}

\begin{proof}
	By (\ref{prodconnect}),
	$$V_{\phi_1}\langle\overline{\nabla}_{V_{\phi_2}}V_{\phi_3},V_{\phi_4}\rangle_\mu= V_{\phi_1}F_{\langle\nabla\phi_2,\nabla^2\phi_3(\nabla\phi_4)\rangle}(\mu)=F_{\langle \nabla\phi_1,\nabla\langle\nabla\phi_2,\nabla^2\phi_3(\nabla\phi_4)\rangle\rangle}(\mu).$$
    Analogously,
    $$V_{\phi_2}\langle\overline{\nabla}_{V_{\phi_1}}V_{\phi_3},V_{\phi_4} \rangle_\mu=F_{\langle \nabla\phi_2,\nabla\langle\nabla\phi_1,\nabla^2\phi_3(\nabla\phi_4)\rangle\rangle}(\mu).$$

        From Lemma \ref{Lemma4}, (\ref{defT}) and the definition of the metric of $P^\infty(M)$,
    \begin{align*}
    \langle\overline{\nabla}_{V_{\phi_2}}V_{\phi_3},\overline{\nabla}_{V_{\phi_1}}V_{\phi_4}\rangle_\mu &=\langle dG_\mu d_\mu^*\flat(\nabla^2\phi_3(\nabla \phi_2)), dG_\mu d_\mu^*\flat(\nabla^2\phi_4(\nabla \phi_1))\rangle_{L^2}\\
	&=\langle \Pi_\mu\flat(\nabla^2\phi_3(\nabla \phi_2)), \Pi_\mu\flat(\nabla^2\phi_4(\nabla \phi_1))\rangle_{L^2}\\
    &= \langle \flat(\nabla^2\phi_3(\nabla \phi_2) ),\flat(\nabla^2\phi_4(\nabla \phi_1))\rangle_{L^2}-\langle T_{\phi_2\phi_3},T_{\phi_1\phi_4}\rangle\\
    & =\int_M\langle \nabla^2\phi_3(\nabla \phi_2),\nabla^2\phi_4(\nabla \phi_1)\rangle d\mu-\langle T_{\phi_2\phi_3},T_{\phi_1\phi_4}\rangle \\
    &=F_{\langle \nabla^2\phi_3(\nabla \phi_2),\nabla^2\phi_4(\nabla \phi_1)\rangle}(\mu)-\langle T_{\phi_2\phi_3},T_{\phi_1\phi_4}\rangle.
    \end{align*}

    Similarly,
    $$\langle\overline{\nabla}_{V_{\phi_1}}V_{\phi_3},\overline{\nabla}_{V_{\phi_2}}V_{\phi_4}\rangle_\mu=F_{\langle \nabla^2\phi_3(\nabla \phi_1),\nabla^2\phi_4(\nabla \phi_2)\rangle}(\mu)-\langle T_{\phi_1\phi_3},T_{\phi_2\phi_4}\rangle$$

    Let us now calculate $\langle \overline{\nabla}_{[V_{\phi_1},V_{\phi_2}]}V_{\phi_3},V_{\phi_4}\rangle_\mu$. By (\ref{eq:bracket-potential}) we may write $[V_{\phi_1},V_{\phi_2}]=V_{\theta_\mu}$ with $$\theta_\mu=G_\mu d_\mu^*\flat(\nabla^2\phi_2(\nabla\phi_1)-\nabla^2\phi_1(\nabla\phi_2)).$$
    Thus, by (\ref{prodconnect}),
    \begin{align*}
        \langle \overline{\nabla}_{[V_{\phi_1},V_{\phi_2}]}V_{\phi_3},V_{\phi_4}\rangle_\mu & =\int_M\langle \nabla\theta_\mu,\nabla^2\phi_3(\nabla\phi_4)\rangle d\mu=\langle d\theta_\mu,\flat(\nabla^2\phi_3(\nabla\phi_4))\rangle_{L^2}\\
    &=\langle dG_\mu d_\mu^*\flat(\nabla^2\phi_2(\nabla\phi_1)-\nabla^2\phi_1(\nabla\phi_2)),\flat(\nabla^2\phi_3(\nabla\phi_4))\rangle_{L^2}\\
    &=\langle \Pi_\mu\flat(\nabla^2\phi_2(\nabla\phi_1)-\nabla^2\phi_1(\nabla\phi_2)),\Pi_\mu\flat(\nabla^2\phi_3(\nabla\phi_4))\rangle_{L^2}\\
        &=\int_M\langle\nabla^2\phi_2(\nabla\phi_1)-\nabla^2\phi_1(\nabla\phi_2),\nabla^2\phi_3(\nabla\phi_4)\rangle d\mu-\langle T_{\phi_1\phi_2},T_{\phi_4\phi_3}\rangle\\
    &+\langle T_{\phi_2\phi_1},T_{\phi_4\phi_3}\rangle\\
    &=\int_M\langle\nabla^2\phi_2(\nabla\phi_1)-\nabla^2\phi_1(\nabla\phi_2),\nabla^2\phi_3(\nabla\phi_4)\rangle d\mu+2\langle T_{\phi_1\phi_2},T_{\phi_3\phi_4}\rangle.
    \end{align*}

    The theorem then follows from the formula
    \begin{align*}
        \langle\overline{R}(V_{\phi_1},V_{\phi_2})V_{\phi_3},V_{\phi_4}\rangle_\mu
    &=V_{\phi_1}\langle\overline{\nabla}_{V_{\phi_2}}V_{\phi_3},V_{\phi_4}\rangle_\mu
    -\langle\overline{\nabla}_{V_{\phi_2}}V_{\phi_3},\overline{\nabla}_{V_{\phi_1}}V_{\phi_4}\rangle_\mu
    \\
    &-V_{\phi_2}\langle\overline{\nabla}_{V_{\phi_1}}V_{\phi_3},V_{\phi_4}\rangle_\mu+\langle\overline{\nabla}_{V_{\phi_1}}V_{\phi_3},\overline{\nabla}_{V_{\phi_2}}V_{\phi_4}\rangle_\mu\\
    &-\langle\overline{\nabla}_{[V_{\phi_1},V_{\phi_2}]}V_{\phi_3},V_{\phi_4}\rangle_\mu,
    \end{align*}
    with the convention $R(X,Y)=\nabla_X\nabla_Y-\nabla_Y\nabla_X-\nabla_{[X,Y]}$, both on $M$ and on $P^\infty(M)$. Substituting the five expressions above, the terms in $T$ give the last three terms of the statement, and it remains to show that the remaining integrands add up to $\langle R(\nabla\phi_1,\nabla\phi_2)\nabla\phi_3,\nabla\phi_4\rangle$. Write $X_i=\nabla\phi_i$ and let $H_i=\nabla^2\phi_i$, viewed as a symmetric endomorphism, so that $\nabla_YX_i=H_iY$. By the Leibniz rule,
    $$\langle X_1,\nabla\langle X_2,H_3X_4\rangle\rangle=\langle H_2X_1,H_3X_4\rangle+\langle(\nabla_{X_1}H_3)X_4,X_2\rangle+\langle H_3X_2,H_4X_1\rangle,$$
    and similarly with $1$ and $2$ exchanged. The terms $\langle H_3X_2,H_4X_1\rangle$ and $\langle H_3X_1,H_4X_2\rangle$ cancel against the second and the fourth terms of the formula, and $\langle H_2X_1-H_1X_2,H_3X_4\rangle$ cancels against the fifth. What remains is
    $$\langle(\nabla_{X_1}H_3)X_4,X_2\rangle-\langle(\nabla_{X_2}H_3)X_4,X_1\rangle=(\nabla_{X_1}\nabla^2\phi_3)(X_2,X_4)-(\nabla_{X_2}\nabla^2\phi_3)(X_1,X_4),$$
    by the symmetry of $\nabla^2\phi_3$. By the Ricci identity for the one-form $d\phi_3$, this equals
    $$-d\phi_3\big(R(X_1,X_2)X_4\big)=\langle R(X_1,X_2)X_3,X_4\rangle .$$

\end{proof}

\begin{corollary}
Fix $\mu\in P^\infty(M)$ and let $\phi_1,\phi_2\in C^\infty(M)$ be such that $\langle V_{\phi_i},V_{\phi_j}\rangle_\mu=\delta_{ij}$ for $i,j=1,2$. Then the sectional curvature of $P^\infty(M)$ at $\mu$ of the plane spanned by $V_{\phi_1},V_{\phi_2}$ is
$$\overline{K}(V_{\phi_1},V_{\phi_2})=\int_M K(\nabla\phi_1,\nabla\phi_2)\left(|\nabla\phi_1|^2|\nabla\phi_2|^2-\langle\nabla\phi_1,\nabla\phi_2\rangle^2\right)d\mu+3|T_{\phi_1\phi_2}|^2_\mu,$$
in which $K$ denotes the sectional curvature of $M$.
\end{corollary}

\begin{corollary}
	If $M$ has nonnegative sectional curvature, then so does $P^\infty(M)$.
\end{corollary}

\begin{proof}
	Immediate from the previous corollary, since $|\nabla\phi_1|^2|\nabla\phi_2|^2-\langle\nabla\phi_1,\nabla\phi_2\rangle^2\geq0$ by Cauchy-Schwarz and $|T_{\phi_1\phi_2}|^2_\mu\geq0$.
\end{proof}

%%%%%%%%%%%%%%%%%%%%%%%%%%%%%%%%%%%%%%%%%%%%%%%%%%%%%%%%%%%%%%%%%%%%%%%%%%%%%%%%%%%%%%%%%%%%%%%%%%%%%%%%%%%%%%%%%%%%%%%%%%%%%%%%%%%%%%%%%%%%%%%%%%%%%%%%%%%%%%%%%%%%%%%%%%%%%%%%%%%%%%%%%%%%%%%%%%%%%%%%%%%%%%%%%%%%%%%

\section{The Riemannian Geometry of \texorpdfstring{$P^\infty(S^1)$}{P(S1)}}\label{secExample}

We now apply Theorem \ref{thm:parallelization} to $M=S^1$, where the Fourier basis provides the global frame. We compute the matrix coefficients of the metric, the connection coefficients, the geodesic equation and the curvature, and we show that $P^\infty(S^1)$ is not flat.

\subsection{Riemannian Metric and its Matrix Coefficients}

Here, we specialize the constructions of the previous sections to $M=S^1$ with the round metric, writing $\mu=\rho\,d\theta$ with $\int_0^{2\pi}\rho\,d\theta=1$.

For the circle, the basis of $C^\infty(S^1)$ given by the Peter-Weyl theorem is the Fourier basis $\{c_n,s_n:n=1,2,\cdots\}$, with $c_n(\theta):=\cos(n\theta)$ and $s_n(\theta):=\sin(n\theta)$, together with the constants, which are quotiented out in $T_\mu P^\infty(S^1)$. It consists of eigenfunctions of the Laplacian, $c_n''=-n^2c_n$ and $s_n''=-n^2s_n$, and together with the constants it is an orthogonal basis of $L^2(d\theta)$, so it satisfies the hypothesis of Theorem \ref{thm:parallelization}. Moreover, if $\psi\in C^\infty(S^1)$ has Fourier coefficients $\widehat{\psi}(n)$, repeated integration by parts gives $|\widehat{\psi}(n)|=O(n^{-k})$ for every $k$, so that the Fourier series of $\psi$ converges in every $C^k$ norm. Hence $\{V_{c_n},V_{s_n}\}_{n\geq1}$ is a global frame for $TP^\infty(S^1)$, and it is direct to compute the coefficients of the Riemannian metric in it.

Throughout this section every tangent vector is written as $V_\psi$ with $\psi\in C^\infty(S^1)$, and the rapid decay just recorded is what makes the series below meaningful. The coefficients appearing in the geodesic equation and in the curvature formula grow polynomially in the frequencies, so that all the series in question converge absolutely for smooth potentials. None of them is continuous for the $L^2(\mu)$ norm, and accordingly the formulas that follow are statements about smooth potentials rather than about the completion of $T_\mu P^\infty(S^1)$.

\begin{prop}\label{matrixcoefficients}
	Fix $\mu\in P^\infty(S^1)$ with density $\rho\in C^\infty(S^1)$, and write its Fourier series as
	\begin{equation}\label{Fourierrho}
		\rho(\theta)=\frac{1}{2\pi}+\sum_{k=1}^\infty\big(a_kc_k(\theta)+b_ks_k(\theta)\big).
	\end{equation}
	Set $A_0:=1$, $B_0:=0$ and, for $j\geq1$, $A_j:=\int_0^{2\pi}c_j\,d\mu=\pi a_j$ and $B_j:=\int_0^{2\pi}s_j\,d\mu=\pi b_j$. Then the matrix coefficients of the Riemannian metric of $P^\infty(S^1)$ at $\mu$ are
	\begin{align}
		\label{metric}
		\begin{split}
			\langle V_{c_n},V_{c_{m}}\rangle_\mu &= \frac{nm}{2}\big(A_{|n-m|}-A_{n+m}\big),
			\\
			\langle V_{s_n},V_{s_m}\rangle_\mu &= \frac{nm}{2}\big(A_{|n-m|}+A_{n+m}\big),\\
			\langle V_{c_n},V_{s_m}\rangle_\mu &= -\frac{nm}{2}\big(B_{n+m}+\operatorname{sgn}(n-m)B_{|n-m|}\big).
		\end{split}
	\end{align}
	In particular the matrix is diagonal only at the uniform measure $\vol$, where it reduces to $\delta_{nm}n^2/2$, with $\delta_{nm}$ the Kronecker delta.
\end{prop}

\begin{proof}
	Since $\nabla c_n=-n\,s_n$ and $\nabla s_n=n\,c_n$, we have $\langle V_{c_n},V_{c_m}\rangle_\mu=nm\int_0^{2\pi}s_ns_m\,d\mu$. The product-to-sum identity $s_ns_m=\tfrac12(c_{|n-m|}-c_{n+m})$ then gives the first formula, and the other two follow in the same way from $c_nc_m=\tfrac12(c_{|n-m|}+c_{n+m})$ and $s_nc_m=\tfrac12(s_{n+m}+s_{n-m})$.
\end{proof}

%%%%%%%%%%%%%%%%%%%%%%%%%%%%%%%%%%%%%%%%%%%%%%%%%%%%%%%%%%%%%%%%%%%%%%%%%%%%%%%%%%%%%%%%%%%%%%%%%%%%%%%%%%%%%%%%%%%%%%%%%%%%%%%%%%%%%%%%%%%%%%%%%%%%%%%%%%%%%%%%%%%%%%%%%%%%%%%%%%%%%%%%%%%%%%%%%%%%%%%%%%%%%%%%%%%%%%%%%%%%%%%%%%%%%%%%%%%%%%%%%%%%%%%%%%%%%%%%%%%%%%%%%%%%%%%%%%%%%%%%%%%%%%%%%%%%%%%%%%%%%%%%%%%%%%%%%%%%%%%%%%%%

\subsection{Connection Coefficients}

We now compute, in the Fourier frame, the connection coefficients with lower index. The symbols $\Gamma^k_{ij}$ of (\ref{Christoffel}) are recovered from them by raising the index with $\overline{g}(\mu)$, which is immediate at the uniform measure and is carried out in Corollary \ref{cor:christoffel-vol}.

\begin{prop}\label{prop:connection-coefficients}
		Let $\mu=\rho\,d\theta\in P^\infty(S^1)$ and set $A_p:=\int_0^{2\pi}c_p\,d\mu$ and $B_p:=\int_0^{2\pi}s_p\,d\mu$, so that $A_0=1$, $B_0=0$ and $A_p=\pi a_p$, $B_p=\pi b_p$ for $p\geq1$. For $\varphi,\chi,\eta$ in the Fourier frame $\{c_p,s_p\}_{p\geq1}$, write $\Gamma_{\varphi\chi,\eta}:=\langle\overline{\nabla}_{V_{\varphi}}V_{\chi},V_{\eta}\rangle_\mu$. Then
	$$\Gamma_{\varphi\chi,\eta}=\int_0^{2\pi}\varphi'\,\chi''\,\eta'\,\rho\,d\theta,$$
	a finite combination of the $A_p$ and $B_p$. Explicitly, for frequencies $n$, $m$ and $l$,
	\begin{align*}
		\Gamma_{c_nc_m,c_l} &= -\frac{nlm^2}{4}\big(A_{|n-l-m|}+A_{|n-l+m|}-A_{|n+l-m|}-A_{n+l+m}\big),\\
		\Gamma_{s_nc_m,s_l} &= -\frac{nlm^2}{4}\big(A_{|n-l-m|}+A_{|n-l+m|}+A_{|n+l-m|}+A_{n+l+m}\big),\\
		\Gamma_{c_ns_m,s_l} &= \frac{nlm^2}{4}\big(A_{|n-l-m|}-A_{|n-l+m|}+A_{|n+l-m|}-A_{n+l+m}\big),\\
		\Gamma_{s_ns_m,c_l} &= \frac{nlm^2}{4}\big(A_{|l-n-m|}-A_{|l-n+m|}+A_{|n+l-m|}-A_{n+l+m}\big),
	\end{align*}
	and the four remaining families of index types are obtained in the same way, with the $B_p$ in place of the $A_p$.
\end{prop}

\begin{proof}
	The first identity is (\ref{prodconnect}) in dimension one, where $\nabla^2\chi(\nabla\eta)=\chi''\eta'$. Every element of the Fourier frame of frequency $m$ satisfies $\chi''=-m^2\chi$, so it remains to expand $\varphi'\eta'\chi$ by the product-to-sum formulas and integrate term by term against $\mu$. For instance, $c_n'c_l'=nl\,s_ns_l=\frac{nl}{2}(c_{n-l}-c_{n+l})$, and multiplying by $c_m$ and integrating gives the first identity. The others are analogous.
\end{proof}

\begin{corollary}\label{cor:geodesic-lower}
	Index the Fourier frame as $\{\phi_i\}_{i\geq1}$, so that each $\phi_i$ is some $c_p$ or $s_p$, and write $\overline{g}_{ij}(\mu)=\langle V_{\phi_i},V_{\phi_j}\rangle_\mu$ and $\Gamma_{ij,k}=\Gamma_{\phi_i\phi_j,\phi_k}$. A smooth curve $\mu_t\in P^\infty(S^1)$ whose velocity field is $V_{\psi_t}$ with $\psi_t=\sum_i\psi_i(t)\phi_i$ is a geodesic if and only if
	$$\sum_l\overline{g}_{kl}(\mu_t)\,\psi_l'(t)+\sum_{i,j}\psi_i(t)\psi_j(t)\,\Gamma_{ij,k}=0$$
	for every $k$, where $\overline{g}_{kl}$ is given by Proposition \ref{matrixcoefficients} and $\Gamma_{ij,k}$ by Proposition \ref{prop:connection-coefficients}.
\end{corollary}

This is the form of (\ref{geochristoffel}) that does not require inverting $\overline{g}(\mu)$. Both the metric coefficients and the connection coefficients appearing in it are given in closed form by Propositions \ref{matrixcoefficients} and \ref{prop:connection-coefficients}. They are evaluated at $\mu_t$, so this system is closed only together with the continuity equation $\partial_t\rho_t=-(\rho_t\psi_t')'$, which determines $\mu_t$ from $\psi_t$. By Corollary \ref{cor:geochristoffel} the dependence on $\mu_t$ can also be removed altogether, and the system then reduces to the Hamilton-Jacobi equation.

\begin{corollary}\label{cor:christoffel-vol}
	At $\mu=\vol$ the matrix $\overline{g}$ is diagonal, with $\overline{g}_{c_nc_n}=\overline{g}_{s_ns_n}=n^2/2$, so that $\Gamma^k_{ij}=2\Gamma_{ij,k}/n_k^2$. Only four families of index types are non-zero, and with $n=n_i$, $m=n_j$, $l=n_k$ they are
	\begin{align*}
		\Gamma^{c_l}_{c_nc_m} &= -\frac{nm^2}{2l}\big(\delta_{|n-l|,m}-\delta_{n+l,m}\big),\\
		\Gamma^{s_l}_{s_nc_m} &= -\frac{nm^2}{2l}\big(\delta_{|n-l|,m}+\delta_{n+l,m}\big),\\
		\Gamma^{s_l}_{c_ns_m} &= \frac{nm^2}{2l}\big(\delta_{n+l,m}+\operatorname{sgn}(n-l)\,\delta_{|n-l|,m}\big),\\
		\Gamma^{c_l}_{s_ns_m} &= \frac{nm^2}{2l}\big(\delta_{n+l,m}+\operatorname{sgn}(l-n)\,\delta_{|n-l|,m}\big).
	\end{align*}
	In particular $\Gamma^k_{ij}=0$ unless $n_k=n_i+n_j$ or $n_k=|n_i-n_j|$, and in that case $|\Gamma^k_{ij}|=n_in_j^2/(2n_k)$.
\end{corollary}

\subsection{The Geodesic Equation}

For $S^1$, the geodesic equation of Corollary \ref{corollarygeodesic} is the inviscid Burgers equation for the velocity field.

\begin{prop}\label{prop:geodesic-circle}
	Let $\mu_t=\rho_t\,d\theta$ be a smooth curve in $P^\infty(S^1)$ with $\dot{\mu}_t=V_{\psi_t}$, and put $u(t,\theta):=\partial_\theta\psi_t(\theta)$. Then $\rho$ satisfies the continuity equation
	$$\partial_t\rho+\partial_\theta(\rho\,u)=0.$$
	The curve is a geodesic if and only if $u$ solves the inviscid Burgers equation $\partial_tu+u\,\partial_\theta u=0$. In that case $u(t,\theta)=u_0(\theta-t\,u(t,\theta))$ with $u_0:=u(0,\cdot)$, and if $\psi_0$ is not constant the solution remains smooth exactly up to the time $-1/\min_\theta u_0'(\theta)$.
\end{prop}

\begin{proof}
	The continuity equation is (\ref{derivativelaplace}) in dimension one, and it holds along every curve with velocity $V_{\psi_t}$, geodesic or not. By Corollary \ref{corollarygeodesic}, $\mu_t$ is a geodesic if and only if $\partial_t\psi_t+\frac12(\partial_\theta\psi_t)^2=0$ after normalizing constants, and differentiating in $\theta$ gives the Burgers equation. Conversely, if $u$ solves the Burgers equation then $\partial_\theta\big(\partial_t\psi_t+\frac12u^2\big)=0$, so that $\partial_t\psi_t+\frac12u^2$ depends only on $t$ and can be absorbed into $\psi_t$. The implicit formula is the method of characteristics, which are the lines $\theta=\theta_0+t\,u_0(\theta_0)$. They remain disjoint exactly as long as $1+t\,u_0'>0$. Since $u_0'=\psi_0''$ has zero mean, it takes negative values unless $\psi_0$ is constant, and the first crossing occurs at $t=-1/\min u_0'$.
\end{proof}

%%%%%%%%%%%%%%%%%%%%%%%%%%%%%%%%%%%%%%%%%%%%%%%%%%%%%%%%%%%%%%%%%%%%%%%%%%%%%%%%%%%%%%%%%%%%%%%%%%%%%%%%%%%%%%%%%%%%%%%%%%%%%%%%%%%%%%%%%%%%%%%%%%%%%%%%%%%%%%%%%%%%%%%%%%%%%%%%%%%%%%%%%%%%%%%%%%%%%%%%%%%%%%%%%%%%%%%%%%%%%%%%%%%%%%%%%%%%%%%%%%%%%%%%%%%%%%%%%%%%%%%%%%%%%%%%%%%%%%%%%%%%%%%%%%%%%%%%%%%%%%%%%%%%%

\subsection{Curvature}

We finish with the curvature of $P^\infty(S^1)$. The orthogonal complement of $\overline{\operatorname{Im}(d)}$ turns out to be one-dimensional here, so that the tensor $T_{\phi\psi}$ of (\ref{defT}) takes values in a fixed line. It then factors as a symplectic form on the Fourier coefficients times a fixed one-form, and the sectional curvature of $P^\infty(S^1)$ can be computed on every plane.

\begin{lemma}\label{lem:complement}
	Let $\mu=\rho\,d\theta\in P^\infty(S^1)$. The orthogonal complement of $\overline{\operatorname{Im}(d)}$ in $\Omega^1_{L^2}(S^1,\mu)$ is the line spanned by $\rho^{-1}d\theta$, and for every $f\in C^\infty(S^1)$
	$$(I-\Pi_\mu)(f\,d\theta)=\frac{\int_0^{2\pi}f\,d\theta}{\int_0^{2\pi}\rho^{-1}\,d\theta}\,\rho^{-1}d\theta.$$
	In particular, $f\,d\theta\in\overline{\operatorname{Im}(d)}$ if and only if $\int_0^{2\pi}f\,d\theta=0$.
\end{lemma}

\begin{proof}
	A one-form $f\,d\theta$ is orthogonal to $dg$ for every $g\in C^\infty(S^1)$ if and only if $\int_0^{2\pi}fg'\rho\,d\theta=0$ for every such $g$, that is, if and only if $(f\rho)'=0$, that is, $f=c\rho^{-1}$ for some constant $c$. Since $\rho$ is smooth and positive on the compact $S^1$, we have $\norm{\rho^{-1}d\theta}^2_{L^2(\mu)}=\int_0^{2\pi}\rho^{-1}d\theta<\infty$, so this line is a genuine one-dimensional subspace. 	The projection formula follows from $\langle f\,d\theta,\rho^{-1}d\theta\rangle_{L^2(\mu)}=\int_0^{2\pi}f\,d\theta$.
	\end{proof}
	
	On the circle, the motions that do not move $\mu$ can be written down explicitly. By the discussion following (\ref{defT}), a one-form $a\,d\theta$ lies in $\operatorname{Im}(d)^\perp$ exactly when the velocity field $Z=a\,\partial_\theta$, where $\partial_\theta$ denotes the unit vector field dual to $d\theta$, satisfies $(\rho a)'=0$ in the sense of distributions, that is, when the flux of mass $\rho a$ equals a constant $c$. The complement is therefore spanned by a circulation of mass around the circle with constant flux, faster where there is less mass, which leaves $\mu$ unchanged. Its kinetic energy $\int_0^{2\pi}a^2\,d\mu=c^2\int_0^{2\pi}\rho^{-1}d\theta$ is finite because $\rho$ is positive. On the real line no such circulation is available, since a constant flux $c\neq0$ would have kinetic energy $c^2\int_{\mathbb{R}}\rho^{-1}dx=\infty$, as recorded in the remark following Proposition \ref{prop:metric-curvature}. This is the way in which the topology of $S^1$ enters: a flow of mass can close up on itself.

	Returning to Lemma \ref{lem:complement}, the projection is governed by a single number, the mean of $f$ against $d\theta$, and computing $T$ on the Fourier frame reduces to elementary integrals.
	
	\begin{theo}\label{thm:curvature}
	Let $\mu=\rho\,d\theta\in P^\infty(S^1)$ and write $C_\mu:=\int_0^{2\pi}\rho^{-1}d\theta$. On the Fourier frame, the only non-vanishing values of $T$ are
	$$T_{c_ns_n}=-T_{s_nc_n}=\frac{n^3\pi}{C_\mu}\,\rho^{-1}d\theta,\qquad n\geq1,$$
	while $T_{c_nc_m}=T_{s_ns_m}=0$ for all $n,m$ and $T_{c_ns_m}=0$ for $n\neq m$.
\end{theo}

\begin{proof}
	In dimension one $\nabla^2\psi(\nabla\phi)=\psi''\phi'$, so by Lemma \ref{lem:complement} the form $T_{\phi\psi}$ is the multiple of $\rho^{-1}d\theta$ with coefficient $\big(\int_0^{2\pi}\psi''\phi'\,d\theta\big)/C_\mu$. This is an integral against $d\theta$, not against $\mu$, so the ordinary orthogonality relations of the Fourier basis apply. For $\phi=c_n$ and $\psi=s_n$ the integrand is $n^3s_n^2$, of integral $n^3\pi$. For every other pair of frame elements the integrand is of the form $c_pc_q$, $s_ps_q$ or $s_pc_q$ with $p\neq q$, and hence has integral zero. The antisymmetry is Lemma \ref{lem:T-antisym}.
\end{proof}

So $T$ couples the cosine and the sine of one and the same frequency, and nothing else. The same computation applies to arbitrary smooth potentials, and the answer has a shape worth recording: the values all lie on the same line, and the dependence on the two arguments is a symplectic form on the Fourier coefficients.

\begin{theo}\label{thm:curvature-general}
	Let $\mu=\rho\,d\theta\in P^\infty(S^1)$ and set
	$$\alpha_\mu:=\frac{\pi}{C_\mu}\,\rho^{-1}d\theta\in\Omega^1_{L^2}(S^1,\mu),\qquad \omega(u,v):=\sum_{n\geq1}n^3\big(u^c_nv^s_n-u^s_nv^c_n\big),$$
	where $u,v\in C^\infty(S^1)$ have Fourier expansions $u=\sum_n(u^c_nc_n+u^s_ns_n)$ and $v=\sum_n(v^c_nc_n+v^s_ns_n)$, so that the series defining $\omega(u,v)$ converges absolutely by the rapid decay of the Fourier coefficients of a smooth function. Then $T_{uv}=\omega(u,v)\,\alpha_\mu$, and consequently, for every plane in $T_\mu P^\infty(S^1)$ spanned by $V_u$ and $V_v$,
	\begin{equation}\label{eq:curvature-plane}
		\overline{K}\big(V_u,V_v\big)=\frac{3\pi^2\,\omega(u,v)^2}{C_\mu\left(\norm{V_u}_\mu^2\norm{V_v}_\mu^2-\langle V_u,V_v\rangle_\mu^2\right)}.
	\end{equation}
\end{theo}

\begin{proof}
	As in the proof of Theorem \ref{thm:curvature}, Lemma \ref{lem:complement} gives
	$$T_{uv}=\frac{1}{C_\mu}\left(\int_0^{2\pi}u'v''\,d\theta\right)\rho^{-1}d\theta .$$
	The Fourier series of $u'=\sum_nn(u^s_nc_n-u^c_ns_n)$ and $v''=-\sum_nn^2(v^c_nc_n+v^s_ns_n)$ converge in $L^2(d\theta)$, so by Parseval's identity
	$$\int_0^{2\pi}u'v''\,d\theta=\pi\sum_{n\geq1}n^3\big(u^c_nv^s_n-u^s_nv^c_n\big)=\pi\,\omega(u,v),$$
	and hence $T_{uv}=\omega(u,v)\,\alpha_\mu$.
	
	For the curvature, recall from the corollary to Theorem \ref{riemanniancurvature} that $\overline{K}=3\norm{T}^2_\mu$ for an orthonormal pair spanning the plane, the first term vanishing here for two independent reasons: $S^1$ is flat, and, being one-dimensional, its gradients are everywhere proportional, so that the Gram factor $|\nabla\phi_1|^2|\nabla\phi_2|^2-\langle\nabla\phi_1,\nabla\phi_2\rangle^2$ vanishes pointwise. Since $\omega$ is alternating and the Gram determinant is quadratic in the same sense, both the numerator and the denominator of (\ref{eq:curvature-plane}) are multiplied by $\det^2$ under a change of basis of the plane, so the right hand side depends only on the plane and it suffices to check it for an orthonormal pair, where it reduces to $3\omega(u,v)^2\norm{\alpha_\mu}^2_\mu$. Finally $\norm{\alpha_\mu}^2_\mu=(\pi/C_\mu)^2\int_0^{2\pi}\rho^{-2}\rho\,d\theta=\pi^2/C_\mu$, which gives (\ref{eq:curvature-plane}).
\end{proof}

Two consequences are immediate. The first identifies the flat directions, and shows that flatness is a linear-algebraic condition on the plane rather than a metric one.

\begin{corollary}\label{cor:isotropic}
	A plane in $T_\mu P^\infty(S^1)$ has vanishing sectional curvature if and only if it is isotropic for the symplectic form $\omega$. In particular, the plane spanned by any two frame fields of distinct frequencies is flat.
\end{corollary}

\begin{proof}
	By (\ref{eq:curvature-plane}) the curvature vanishes if and only if $\omega(u,v)=0$, since $C_\mu$ is finite and the Gram determinant of a genuine plane is strictly positive. For two frame fields of distinct frequencies, $\omega$ vanishes because it only couples indices of equal frequency.
\end{proof}

The second is the opposite extreme, the planes on which $\omega$ is as large as it can be for a given pair of frequencies.

\begin{corollary}\label{cor:curvature-fourier}
	For every $n\geq1$,
	\begin{equation}\label{eq:curvature-general}
		\overline{K}(V_{c_n},V_{s_n})=\frac{12\pi^2n^2}{C_\mu\left(1-|\widehat{\mu}(2n)|^2\right)},
	\end{equation}
	where $\widehat{\mu}(k):=\int_0^{2\pi}e^{ik\theta}\,d\mu=\pi(a_k+ib_k)$. In particular $P^\infty(S^1)$ is not flat.
\end{corollary}

\begin{proof}
	Here $\omega(c_n,s_n)=n^3$, and by Proposition \ref{matrixcoefficients}, since $A_0=1$ and only $A_{2n}$ and $B_{2n}$ occur,
	$$\norm{V_{c_n}}^2_\mu\norm{V_{s_n}}^2_\mu-\langle V_{c_n},V_{s_n}\rangle^2_\mu=\frac{n^4}{4}\left(1-|\widehat{\mu}(2n)|^2\right).$$
	Substituting in (\ref{eq:curvature-plane}) gives (\ref{eq:curvature-general}), which is positive because $\rho$ is smooth and positive on the compact $S^1$, so that $C_\mu<\infty$.
\end{proof}

Formula (\ref{eq:curvature-general}) invites the question of how large the curvature can be as $\mu$ varies. The two factors in its denominator pull in opposite directions, since a density far from uniform makes $C_\mu$ large but may also make $|\widehat{\mu}(2n)|$ large. The following sharpening of the Cauchy-Schwarz inequality settles the competition.

\begin{lemma}\label{lem:sharp-cs}
	Let $\mu=\rho\,d\theta\in P^\infty(S^1)$ and let $k\geq1$. Then
	$$\left(\int_0^{2\pi}\rho^{-1}d\theta\right)\left(1-|\widehat{\mu}(k)|^2\right)\geq4\pi^2,$$
	with equality if and only if $\rho\equiv1/2\pi$.
\end{lemma}

\begin{proof}
	Write $m:=\widehat{\mu}(k)$ and set $g(\theta):=1-m\,e^{-ik\theta}$. Since $k\geq1$ we have $\int_0^{2\pi}e^{-ik\theta}d\theta=0$, so $\int_0^{2\pi}g\,d\theta=2\pi$. Moreover, using $\int\rho\,d\theta=1$ and $\int e^{-ik\theta}\rho\,d\theta=\overline{m}$,
	$$\int_0^{2\pi}|g|^2\rho\,d\theta=1-m\overline{m}-\overline{m}m+|m|^2=1-|m|^2.$$
	Applying the Cauchy-Schwarz inequality to the factorization $g=(g\sqrt\rho)\cdot\rho^{-1/2}$ yields
	$$4\pi^2=\left|\int_0^{2\pi}g\,d\theta\right|^2\leq\left(\int_0^{2\pi}|g|^2\rho\,d\theta\right)\left(\int_0^{2\pi}\rho^{-1}d\theta\right)=\left(1-|m|^2\right)\int_0^{2\pi}\rho^{-1}d\theta,$$
	which is the stated inequality. Equality holds if and only if $g\rho$ is constant, that is, $\rho=c/(1-m\,e^{-ik\theta})$ for some $c$. Integrating $g\rho=c$ over $[0,2\pi]$ gives $2\pi c=1-|m|^2$, so $c$ is real, and $c\neq0$ because $\rho>0$. Writing $m=|m|e^{i\alpha}$, the right hand side is real for every $\theta$ only if $|m|\sin(\alpha-k\theta)\equiv0$, which forces $m=0$ and hence $\rho$ constant. Conversely $\rho\equiv1/2\pi$ gives $m=0$ and $\int\rho^{-1}d\theta=4\pi^2$.
\end{proof}

The uniform measure therefore maximizes the curvature in every one of these planes.

\begin{corollary}\label{cor:curvature-bound}
	For every $\mu\in P^\infty(S^1)$ and every $n\geq1$,
	$$0<\overline{K}(V_{c_n},V_{s_n})\leq3n^2,$$
	with equality on the right if and only if $\mu=\vol$. Moreover, along any family of measures whose densities degenerate, in the sense that $\int_0^{2\pi}\rho^{-1}d\theta\to\infty$, the sectional curvature of every such plane tends to $0$.
\end{corollary}

\begin{proof}
	The bound is immediate from (\ref{eq:curvature-general}) and Lemma \ref{lem:sharp-cs}, and positivity from $C_\mu<\infty$. For the last assertion, $1-|\widehat{\mu}(2n)|^2\leq1$ gives $\overline{K}(V_{c_n},V_{s_n})\leq12\pi^2n^2/C_\mu$, which tends to $0$ as $C_\mu\to\infty$.
\end{proof}

The picture that emerges is that $P^\infty(S^1)$ has non-negative sectional curvature, as it must by the corollary to Theorem \ref{riemanniancurvature} since $S^1$ is flat, but that the curvature is unbounded above as the frequency grows, is maximized at the uniform measure in each Fourier plane, and degenerates to zero as $C_\mu$ grows without bound.

Everything above was computed inside the formal Riemannian structure of Otto. Since the conclusion runs against the expectation that a Wasserstein space is flat whenever its base manifold is one-dimensional, we close by verifying it metrically. The fourth order defect of $W_2$ along two geodesics is computed directly from the definition of $W_2$, and it agrees with the value predicted by Otto's metric. A purely metric consequence, that no neighbourhood of $\vol$ is CAT(0), is recorded after the proof. The Remark following the proof explains, in the terms of Lemma \ref{lem:complement}, why the real line behaves differently.

Recall that in a Riemannian manifold the sectional curvature is encoded in the fourth order behaviour of the distance between two geodesics leaving a common point: for tangent vectors $u_1,u_2$ at a point $p$,
$$d^2\big(\exp_p(tu_1),\exp_p(tu_2)\big)=t^2\norm{u_1-u_2}^2-\frac{t^4}{3}\big\langle R(u_1,u_2)u_2,u_1\big\rangle+O(t^5),$$
where $R$ denotes the curvature tensor of the manifold in question. At $\vol$ the curve $t\mapsto(\mathrm{Id}+tv)_*\vol$ is, for $t$ small enough that $\mathrm{Id}+tv$ is increasing, the geodesic with initial velocity $v$, by displacement interpolation. The proposition below therefore evaluates the left hand side of that expansion directly from the definition of $W_2$. Its factor $3/t^4$ is the normalization that cancels the $1/3$ above, while the factor $3$ on the right hand side is the one carried by $\overline{K}=3\norm{T}^2_\mu$, so that the equality of the two sides is the meeting of two independent computations and not a simplification.

\begin{prop}\label{prop:metric-curvature}
	Let $\vol=d\theta/2\pi$ be the uniform measure on $S^1$, let $\phi,\psi\in C^\infty(S^1)$ and put $v:=\nabla\phi$, $w:=\nabla\psi$. Then
	$$\lim_{t\to0}\frac{3}{t^4}\left[t^2\norm{v-w}^2_{L^2(\vol)}-W_2^2\big((\mathrm{Id}+tv)_*\vol,\ (\mathrm{Id}+tw)_*\vol\big)\right]=3\left(\int_{S^1} v\,w'\,d\vol\right)^2,$$
	and the right hand side equals $3\norm{T_{\phi\psi}}^2_{\vol}$. The left hand side is computed from the definition of $W_2$ alone. If the expansion recalled above is applied to Otto's metric, the left hand side is $\big\langle\overline{R}(V_\phi,V_\psi)V_\psi,V_\phi\big\rangle_{\vol}$, and dividing it by the Gram determinant $\norm{V_\phi}^2_{\vol}\norm{V_\psi}^2_{\vol}-\langle V_\phi,V_\psi\rangle^2_{\vol}$ recovers $(\ref{eq:curvature-plane})$ at $\mu=\vol$. For $\phi=c_n$ and $\psi=s_n$ that quotient equals $3n^2$, the value given by Corollary \ref{cor:curvature-bound} at $\mu=\vol$, so that the non-flatness of $P^\infty(S^1)$ and the growth of its curvature with the frequency are already visible from the definition of $W_2$.
\end{prop}

\begin{proof}
	The idea is elementary. Both measures are images of $\vol$ under maps close to the identity. On the real line there would be nothing to decide, since the monotone map is optimal and its cost is exactly $t^2\norm{v-w}^2_{L^2(\vol)}$. On the circle that map may be composed with a rigid rotation, so there is a whole family of admissible costs, one for each rotation $\beta$, and $W_2^2$ is the smallest of them. Write $C(\beta)$ for the cost of the rotation by $\beta$, so that $C(0)$ is the value just quoted.

	Everything then depends on whether $\beta=0$ is optimal. It is not, as soon as $C'(0)\neq0$, and the gain is easy to estimate. Near the origin $C$ is a parabola,
	$$C(\beta)\approx C(0)+C'(0)\,\beta+\tfrac12C''(0)\,\beta^2,$$
	whose minimum value is $C(0)-C'(0)^2/2C''(0)$. Completing this square is the whole proof, and the two derivatives are computed in Step 2.

	The order in $t$ is decided there as well, and may be announced now. We will find $C''(0)=2+O(t^2)$, so the denominator is harmless. The term of order $t$ in $C'(0)$, on the other hand, is a multiple of $\int_{S^1}(v-w)\,d\vol$, which vanishes because $v$ and $w$ are derivatives of periodic functions. Hence $C'(0)$ is of order $t^2$, its square is of order $t^4$, and this is the fourth order defect of the statement.

	\emph{Step 1: the cost of a rotation.} Cut the circle at $\theta=0$ and let
	$$F_\mu(\theta):=\mu\big([0,\theta]\big)=\int_0^\theta\rho(\sigma)\,d\sigma\qquad(0\leq\theta<2\pi),$$
	extended to $\mathbb{R}$ by $F_\mu(\theta+2\pi)=F_\mu(\theta)+1$. It is smooth and strictly increasing, and we read its inverse in arclength by setting
	$$\widetilde{Q}_\mu(\theta):=F_\mu^{-1}\big(\theta/2\pi\big),$$
	the quantile function of $\mu$ in this parametrization, which satisfies $\widetilde{Q}_\mu(\theta+2\pi)=\widetilde{Q}_\mu(\theta)+2\pi$ and $\widetilde{Q}_{\vol}=\mathrm{Id}$. Equivalently, $\widetilde{Q}_\mu$ is the increasing map commuting with $\theta\mapsto\theta+2\pi$ that pushes $\vol$ forward to $\mu$, and it is unique up to a translation of its argument, which is the choice of cut point. The only property we use is that it turns a pushforward into the map itself: for $S$ increasing with $S(\theta+2\pi)=S(\theta)+2\pi$,
	$$\widetilde{Q}_{S_*\vol}=S .$$
	In particular $\widetilde{Q}_\mu=\mathrm{Id}+tv$ and $\widetilde{Q}_\nu=\mathrm{Id}+tw$, for $t$ small enough that both maps are increasing.

	By \cite{delon2010fast} the optimal coupling on the circle is monotone up to a rotation, and in this parametrization
	$$W_2^2(\mu,\nu)=\min_{\beta\in\mathbb{R}}C(\beta),\qquad C(\beta):=\int_{S^1}\big|\widetilde{Q}_\mu(\theta)-\widetilde{Q}_\nu(\theta+\beta)\big|^2d\vol(\theta),$$
	where $\beta$ is the rotation measured in arclength, the difference is computed on the lift, and $C$ is convex. The ambiguity in the cut point translates the argument of each quantile function, which the minimization over $\beta$ absorbs. Substituting the two quantile functions, the integrand becomes explicit and all trace of the construction disappears:
	$$C(\beta)=\int_{S^1}\big|t\,v(\theta)-\beta-t\,w(\theta+\beta)\big|^2d\vol(\theta).$$
	Write $L:=\max\left(\norm{v}_\infty,\norm{w}_\infty\right)$. Step 4 below checks that every minimizer of $C$ satisfies $|\beta|\leq5tL$, and that on that range the distance computed on the lift is the geodesic distance of $S^1$ once $t<\pi/7L$. All estimates that follow are taken on $|\beta|\leq5tL$.

	\emph{Step 2: the first two derivatives at $\beta=0$.} Differentiating under the integral sign twice,
	$$C'(\beta)=-2\int_{S^1}\big(t\,v(\theta)-\beta-t\,w(\theta+\beta)\big)\big(1+t\,w'(\theta+\beta)\big)\,d\vol(\theta),$$
	$$C''(\beta)=2\int_{S^1}\big(1+t\,w'(\theta+\beta)\big)^2d\vol(\theta)-2t\int_{S^1}\big(t\,v(\theta)-\beta-t\,w(\theta+\beta)\big)\,w''(\theta+\beta)\,d\vol(\theta).$$
	Since $v=\phi'$ and $w=\psi'$ with $\phi$ and $\psi$ periodic, the integral over $S^1$ of a derivative vanishes, which kills four terms at once:
	$$\int_{S^1}v\,d\vol=\int_{S^1}w\,d\vol=\int_{S^1}w'\,d\vol=0,\qquad \int_{S^1}w\,w'\,d\vol=\tfrac12\int_{S^1}\big(w^2\big)'\,d\vol=0 .$$
	Evaluating at $\beta=0$ and using these,
	$$C(0)=t^2\norm{v-w}^2_{L^2(\vol)},\qquad C'(0)=-2t^2\int_{S^1}v\,w'\,d\vol,\qquad C''(0)=2+O(t^2).$$
	The middle identity is the one announced above. Before the cancellations it reads
	$$C'(0)=-2t\int_{S^1}(v-w)\,d\vol-2t^2\int_{S^1}(v-w)\,w'\,d\vol,$$
	and it is the vanishing of the first integral that removes the order $t$ and leaves an order $t^2$. In particular $\beta=0$ is not optimal whenever $\int_{S^1}v\,w'\,d\vol\neq0$, and it is exactly this failure that the circle affords and the line does not.

	\emph{Step 3: the gain.} On $|\beta|\leq5tL$ the integrand and its $\beta$-derivatives are bounded. The function $t\,v(\theta)-\beta-t\,w(\theta+\beta)$ is itself $O(t)$ there, and each of its derivatives in $\beta$ beyond the first carries a factor $t$, so that $C'''=O(t)$ there. Since $C''(0)=2+O(t^2)$, the function $C'$ is strictly increasing on that range for $t$ small, hence has a single zero $\beta^*$, and
	$$\beta^*=-\frac{C'(0)}{C''(0)}+O(t^5)=t^2\int_{S^1}v\,w'\,d\vol+O(t^4).$$
	Taylor's formula at $\beta=0$ with the bound on $C'''$ now gives
	$$W_2^2(\mu,\nu)=C(\beta^*)=C(0)-\frac{C'(0)^2}{2C''(0)}+O(t^5)=t^2\norm{v-w}^2_{L^2(\vol)}-t^4\left(\int_{S^1}v\,w'\,d\vol\right)^2+O(t^5),$$
	which is the stated limit.

	\emph{Step 4: the two claims of Step 1.} Expanding the square and using $\big|t\,v(\theta)-t\,w(\theta+\beta)\big|\leq2tL$ gives $C(\beta)\geq\beta^2-4tL|\beta|$, while $C(0)\leq4t^2L^2$. Hence $C(\beta)>C(0)$ as soon as $\beta^2-4tL|\beta|>4t^2L^2$, which holds for $|\beta|\geq5tL$, so every minimizer satisfies $|\beta^*|\leq5tL$. For such a $\beta$ each particle is displaced by $\big|t\,v(\theta)-\beta-t\,w(\theta+\beta)\big|\leq7tL$, which is less than half a turn once $t<\pi/7L$, and in that regime the distance computed on the lift is the geodesic distance of $S^1$.

	\emph{Step 5: the right hand side.} Here $\int_{S^1}v\,w'\,d\vol=\frac{1}{2\pi}\int_0^{2\pi}\phi'\psi''\,d\theta$, while Lemma \ref{lem:complement} with $\rho\equiv1/2\pi$ gives $C_{\vol}=4\pi^2$ and
	$$T_{\phi\psi}=\frac{\int_0^{2\pi}\psi''\phi'\,d\theta}{4\pi^2}\,\rho^{-1}d\theta=\frac{1}{2\pi}\left(\int_0^{2\pi}\psi''\phi'\,d\theta\right)d\theta,$$
	whose squared $L^2(\vol)$ norm is $\big(\frac{1}{2\pi}\int_0^{2\pi}\psi''\phi'\,d\theta\big)^2$. This is $\left(\int_{S^1}v\,w'\,d\vol\right)^2$, so the limit equals $3\norm{T_{\phi\psi}}^2_{\vol}$.
\end{proof}

\begin{remark}
	Proposition \ref{prop:metric-curvature} has a purely metric consequence: no neighbourhood of $\vol$ in $P_2(S^1)$ is a CAT(0) space. For $t$ small the curves $\gamma_\phi(t):=(\mathrm{Id}+t\nabla\phi)_*\vol$ and $\gamma_\psi(t):=(\mathrm{Id}+t\nabla\psi)_*\vol$ are geodesics issuing from $\vol$, with
	$$W_2(\vol,\gamma_\phi(t))=t\norm{\nabla\phi}_{L^2(\vol)},\qquad W_2(\vol,\gamma_\psi(t))=t\norm{\nabla\psi}_{L^2(\vol)},$$
	and by the proposition $W_2(\gamma_\phi(t),\gamma_\psi(t))/t$ tends to $\norm{\nabla\phi-\nabla\psi}_{L^2(\vol)}$. Hence the comparison angle at $\vol$ of the triangle with vertices $\vol$, $\gamma_\phi(t)$, $\gamma_\psi(t)$ converges, as $t\to0$, to the angle between $\nabla\phi$ and $\nabla\psi$ in $L^2(\vol)$, and this limit is the Alexandrov angle between the two geodesics at $\vol$. In a CAT(0) space the Alexandrov angle between two sides of a geodesic triangle is at most the corresponding angle of its comparison triangle \cite[Proposition II.1.7(4)]{BH99}, so the comparison angle would never be smaller than its limit, that is, $W_2(\gamma_\phi(t),\gamma_\psi(t))\geq t\norm{\nabla\phi-\nabla\psi}_{L^2(\vol)}$ for all small $t$. The proposition gives the strict opposite inequality whenever $\int_{S^1}\phi'\psi''\,d\vol\neq0$, for instance for $\phi=c_n$ and $\psi=s_n$.
\end{remark}

\begin{remark}
	The computation above isolates what distinguishes the circle from the line, and the same dichotomy is visible on both sides of the paper.
	
	On the Riemannian side, the argument of Lemma \ref{lem:complement} applied to $\mathbb{R}$ would again produce $\rho^{-1}dx$ as the only candidate for a one-form orthogonal to all exact ones. But for every probability density on $\mathbb{R}$ one has $\norm{\rho^{-1}dx}^2_{L^2(\mu)}=\int_{\mathbb{R}}\rho^{-1}dx=\infty$, so that candidate does not belong to $\Omega^1_{L^2}(\mathbb{R},\mu)$, the orthogonal complement of $\overline{\operatorname{Im}(d)}$ is trivial, and $T$ vanishes identically. Equivalently, on $\mathbb{R}$ every function is a derivative, whereas on $S^1$ only those of vanishing mean are. Thus $I-\Pi_\mu$ is of rank zero in the first case and of rank one in the second, and this is what $H^1(S^1)\neq0$ contributes.
	
	On the metric side, the same dichotomy appears as the rotation parameter $\beta$ of Proposition \ref{prop:metric-curvature}. On the line the monotone coupling between two absolutely continuous measures is unique, there is nothing to optimize, and the defect vanishes. On the circle the monotone coupling is determined only up to a rotation, and optimizing it produces the fourth order gain.
\end{remark}

\section{The Push-Forward Derivative of an Expanding Circle Map}\label{sec:kloeckner}

The frame of Theorem \ref{thm:parallelization} was built to make the Riemannian structure of $P^\infty(M)$ computable. This short section puts it to work on a computation carried out elsewhere, the derivative of the push-forward action of an expanding circle map at its invariant measure, obtained by Kloeckner in \cite{Klo13} inside the differential structure of Gigli. We do not improve on that theorem and we use it as stated. What the frame contributes is of a different kind, and is of two sorts. It replaces the assertion that every eigenvalue has infinite multiplicity by an explicit parametrization of each eigenspace, so that an eigenvector can be written down. And it is defined simultaneously at every measure with smooth positive density, whereas the computation below takes place at the invariant measure alone. Both points are made precise at the end of the section.

We keep the notation of Section \ref{secExample}, so that $c_m(\theta)=\cos(m\theta)$ and $s_m(\theta)=\sin(m\theta)$, and we write $e_m(\theta)=e^{im\theta}$. Fix an integer $d\geq2$ and let $\Phi_d(\theta)=d\theta\bmod2\pi$ be the standard degree $d$ self-covering map of the circle, acting on measures by push-forward. Each point has $d$ preimages, at each of which $\Phi_d$ expands by $d$, so that $\Phi_{d\#}\vol=\vol$ and $\vol$ is the unique absolutely continuous invariant measure of $\Phi_d$.

The two frameworks describe the same tangent space. In \cite{Klo13} the tangent space at $\mu$ is the closure in $L^2(\mu)$ of the gradients of smooth functions, which is the space $L^2_0(\mu)$ of Subsection \ref{sec1}, and the two inner products coincide. The only difference is one of labelling: there a tangent vector is recorded by its displacement field $v$, here by a potential $\psi$ with $\nabla\psi=v$. On $S^1$ the potential of a field of frequency $m$ has the same frequency and is scaled by $1/m$, and this is the only conversion factor that appears below.

Recall that the Perron-Frobenius, or transfer, operator of $\Phi_d$ is
$$\left(\mathcal{L}_df\right)(x)=\frac1d\sum_{\Phi_d(y)=x}f(y).$$
Its action on exponentials is a shift. Indeed the preimages of $\theta$ are the points $(\theta+2\pi j)/d$ with $0\leq j<d$, so that
$$\left(\mathcal{L}_de_m\right)(\theta)=\frac1d\,e^{im\theta/d}\sum_{j=0}^{d-1}e^{2\pi imj/d}=\begin{cases}e_{m/d}(\theta),&\text{if }d\mid m,\\ 0,&\text{otherwise,}\end{cases}$$
the inner sum being $d$ when $d$ divides $m$ and $0$ otherwise. Equivalently, on the measure side $\widehat{\Phi_{d\#}\mu}(n)=\widehat{\mu}(nd)$, in the notation of Corollary \ref{cor:curvature-fourier}.

\begin{theo}[{\cite[Theorem 1.2]{Klo13}}]\label{thm:kloeckner}
	The map $\Phi_{d\#}$ has a G\^ateaux derivative at $\vol$, equal to $d$ times the Perron-Frobenius operator of $\Phi_d$ acting on $L^2_0(\vol)$. In particular its spectrum is the disc of radius $d$, and every number of modulus smaller than $d$ is an eigenvalue of infinite multiplicity.
\end{theo}

Theorem \ref{thm:kloeckner} describes the derivative as an operator on displacement fields, while the frame of Theorem \ref{thm:parallelization} labels tangent vectors by potentials. By the dictionary recorded above the two tangent spaces are the same, so the theorem applies to the frame vectors directly and only the labels have to be converted. That conversion is not innocuous. Passing from a field of frequency $m$ to its potential divides by $m$, while the derivative sends frequency $m$ to frequency $m/d$, so the two rescalings do not cancel and leave a factor $d$ behind. This is the origin of the $d^2$ below, which is not to be read against the spectral radius $d$ of Theorem \ref{thm:kloeckner}. The two statements are reconciled immediately after the proof.

\begin{prop}\label{prop:kloeckner-frame}
	In the frame $\{V_{c_m},V_{s_m}\}_{m\geq1}$ of Theorem \ref{thm:parallelization},
	$$D_{\vol}\Phi_{d\#}\,V_{c_m}=\begin{cases}d^2\,V_{c_{m/d}},&\text{if }d\mid m,\\ 0,&\text{otherwise,}\end{cases}$$
	and the same formula holds on the $V_{s_m}$.
\end{prop}

\begin{proof}
	The displacement field of $V_{c_m}$ is $c_m'=-m\,s_m$. By Theorem \ref{thm:kloeckner} and the action of $\mathcal{L}_d$ recorded above, its image is $0$ unless $d$ divides $m$, in which case it is $d\cdot\left(-m\,s_{m/d}\right)$. It remains to read that field as a potential. Since $c_{m/d}'=-(m/d)\,s_{m/d}$, we have $-m\,d\,s_{m/d}=d^2\,c_{m/d}'$, so the resulting tangent vector is $d^2V_{c_{m/d}}$. The computation on the $V_{s_m}$ is identical.
\end{proof}

The weight $d^2$ is an artifact of the frame, which is not orthonormal for the metric of Otto: at $\vol$ one has $\norm{V_{c_m}}^2_{\vol}=\int_0^{2\pi}|c_m'|^2d\vol=m^2/2$, so the frame vectors grow linearly in the frequency. Normalizing them, $\widehat{V}_{c_m}:=V_{c_m}/\norm{V_{c_m}}_{\vol}$, the weight becomes
$$d^2\,\frac{\norm{V_{c_{m/d}}}_{\vol}}{\norm{V_{c_m}}_{\vol}}=d^2\cdot\frac1d=d,$$
so that in the orthonormal frame $D_{\vol}\Phi_{d\#}$ is the weighted shift of constant weight $d$ supported on the sublattice of multiples of $d$. Its spectral radius is $d$, in agreement with Theorem \ref{thm:kloeckner}.

In this form the eigenvalues are read off by inspection. Writing $u=\sum_{m\geq1}\big(u_m\widehat{V}_{c_m}+u'_m\widehat{V}_{s_m}\big)$, with $\widehat{V}_{s_m}$ normalized in the same way, the equation $D_{\vol}\Phi_{d\#}u=\zeta u$ amounts to $d\,u_{kd}=\zeta u_k$ and $d\,u'_{kd}=\zeta u'_k$ for every $k\geq1$, that is, to
$$u_{kd^j}=\left(\frac{\zeta}{d}\right)^ju_k,\qquad u'_{kd^j}=\left(\frac{\zeta}{d}\right)^ju'_k\qquad(j\geq0).$$
The coefficients indexed by the frequencies not divisible by $d$, of the cosines and of the sines, may therefore be prescribed freely, and the remaining ones are obtained by propagation along the multiples. For $|\zeta|<d$ the propagation is geometrically decreasing and the resulting vector is square summable, so every such $\zeta$ is an eigenvalue whose eigenspace is parametrized by an infinite family of free coefficients. This is the infinite multiplicity of Theorem \ref{thm:kloeckner}, and the frame has turned it into a statement about a shift.

It is worth being precise about what the frame has and has not done here. The description of $\mathcal{L}_d$ as a shift on exponentials requires no frame at all, and Theorem \ref{thm:kloeckner} has been used rather than reproved, so nothing above improves on \cite{Klo13}. Two things are gained nonetheless. The first is that the eigenspaces are exhibited and not merely counted. The display above is a bijection between the eigenspace of $\zeta$ and the pairs of square summable sequences indexed by the frequencies not divisible by $d$, so that an eigenvector can be produced on demand, whereas infinite multiplicity on its own asserts only that many of them exist. The second is that the Fourier basis of $L^2_0(\vol)$ used above is tied to the invariant measure, being orthonormal there and nowhere else, while the frame of Theorem \ref{thm:parallelization} is defined simultaneously at every $\mu\in P^\infty(S^1)$, with the metric coefficients of Proposition \ref{matrixcoefficients}. Since the density of $\Phi_{d\#}\mu$ is $\mathcal{L}_d\rho$, which is smooth and positive whenever $\rho$ is, the action preserves $P^\infty(S^1)$, and the same matrix description is therefore available away from the fixed point. We do not pursue that here.

	We stress that what has been reformulated is the derivative at $\vol$, and two limitations should be recorded. First, for $\zeta\neq0$ the eigenvectors produced above have coefficients decaying only polynomially in the frequency, since $u$ supported on the frequencies $d^j$ satisfies $|u_{d^j}|=|\zeta/d|^j$. They therefore lie in $L^2_0(\vol)$, as Theorem \ref{thm:kloeckner} asserts, but not in the space of smooth potentials, so they are tangent vectors of the completion rather than of $P^\infty(S^1)$. For $\zeta=0$ the eigenspace is the kernel, spanned by the $V_{c_m}$ and $V_{s_m}$ with $d\nmid m$, and it does contain smooth potentials. Second, the metric mean dimension estimates of \cite[Theorem 1.1]{Klo13} concern the action on all of $P(S^1)$ with the $W_p$ distance, a metric rather than differentiable invariant, and have no counterpart here. The infinitesimal deformations of an invariant measure themselves, on the other hand, are exactly the eigenvectors above, and their study in a broader setting, together with the failure of first order rigidity that they express, is the subject of work in preparation by two of the authors with D. Finamore.

%%%%%%%%%%%%%%%%%%%%%%%%%%%%%%%%%%%%%%%%%%%%%%%%%%%%%%%%%%%%%%%%%%%%%%%%%%%%%%%%%%%%%%%%%%%%%%%%%%%%%%%%%%%%%%%%%%%%%%%%%%%%%%%%%%%%%%%

\section*{Acknowledgements}
\noindent C. S. R. has been partially supported by S\~{a}o Paulo Research Foundation (FAPESP): grant \#2018/13481-0, and grant \#2020/04426-6. L. A. B. San Martin has been partially supported by S\~{a}o Paulo Research Foundation (FAPESP): grant \#2018/13481-0. The opinions, hypotheses and conclusions or recommendations expressed in this work are the responsibility of the authors and do not necessarily reflect the views of FAPESP.

\noindent A.M.S.G, and C. S. R. would like to acknowledge support from the Max Planck Society, Germany, through the award of a Max Planck Partner Group for Geometry and Probability in Dynamical Systems.

%%%%%%%%%%%%%%%%%%%%%%%%%%%%%%%%%%%%%%%%

\bibliographystyle{amsalpha}

\begin{thebibliography}{amsalpha}


\bibitem[AG08]{ambrosiogigli2008}
L.~Ambrosio and N.~Gigli, \emph{Construction of the parallel transport in the Wasserstein space}, Methods Appl. Anal. \textbf{15} (2008), no.~1, 1--30.

\bibitem[AGS05]{ambrosio2005gradient}
L.~Ambrosio, N.~Gigli, and G.~Savar\'e, \emph{Gradient flows in metric spaces and in the space of probability measures}, Lectures in Mathematics ETH Z\"urich, Birkh\"auser, Basel, 2005.

\bibitem[BH99]{BH99}
M.~R.~Bridson and A.~Haefliger, \emph{Metric spaces of non-positive curvature}, Grundlehren der mathematischen Wissenschaften, vol.~319, Springer-Verlag, Berlin, 1999.

\bibitem[Bil99]{billingsley2013convergence}
P.~Billingsley, \emph{Convergence of probability measures}, second ed., Wiley Series in Probability and Statistics, John Wiley \& Sons, New York, 1999.

\bibitem[Cho11]{chodosh2011interval}
O.~Chodosh, \emph{Optimal transport and Ricci curvature: Wasserstein space over the interval}, arXiv:1105.2883, 2011.

\bibitem[DF21]{ding2021geometry}
H.~Ding and S.~Fang, \emph{Geometry on the Wasserstein space over a compact Riemannian manifold}, Acta Math. Sci. Ser. B (Engl. Ed.) \textbf{41} (2021), no.~6, 1959--1984.

\bibitem[DSS10]{delon2010fast}
J.~Delon, J.~Salomon, and A.~Sobolevski, \emph{Fast transport optimization for Monge costs on the circle}, SIAM J. Appl. Math. \textbf{70} (2010), no.~7, 2239--2258.

\bibitem[Gig11]{gigli2011inverse}
N.~Gigli, \emph{On the inverse implication of Brenier-McCann theorems and the structure of $(P_2(M),W_2)$}, Methods Appl. Anal. \textbf{18} (2011), no.~2, 127--158.

\bibitem[GKP11]{GangboKimPacini2011}
W.~Gangbo, H.~K.~Kim, and T.~Pacini, \emph{Differential forms on Wasserstein space and infinite-dimensional Hamiltonian systems}, Mem. Amer. Math. Soc. \textbf{211} (2011), no.~993.

\bibitem[Klo13]{Klo13}
B.~Kloeckner, \emph{Optimal transport and dynamics of expanding circle maps acting on measures}, Ergodic Theory Dynam. Systems \textbf{33} (2013), no.~2, 529--548.

\bibitem[KM97]{kriegl1997convenient}
A.~Kriegl and P.~W.~Michor, \emph{The convenient setting of global analysis}, Mathematical Surveys and Monographs, vol.~53, American Mathematical Society, Providence, RI, 1997.

\bibitem[Lot08]{lott2006some}
J.~Lott, \emph{Some geometric calculations on Wasserstein space}, Comm. Math. Phys. \textbf{277} (2008), no.~2, 423--437.

\bibitem[Lot17a]{lott2017intrinsic}
J.~Lott, \emph{An intrinsic parallel transport in Wasserstein space}, Proc. Amer. Math. Soc. \textbf{145} (2017), no.~12, 5329--5340.

\bibitem[Lot17b]{lott2017ontangent}
J.~Lott, \emph{On tangent cones in Wasserstein space}, Proc. Amer. Math. Soc. \textbf{145} (2017), no.~7, 3127--3136.

\bibitem[LV09]{lott2009ricci}
J.~Lott and C.~Villani, \emph{Ricci curvature for metric-measure spaces via optimal transport}, Ann. of Math. (2) \textbf{169} (2009), no.~3, 903--991.

\bibitem[McC97]{mccann1997aconvexity}
R.~J.~McCann, \emph{A convexity principle for interacting gases}, Adv. Math. \textbf{128} (1997), no.~1, 153--179.

\bibitem[Ott01]{otto2001geometry}
F.~Otto, \emph{The geometry of dissipative evolution equations: the porous medium equation}, Comm. Partial Differential Equations \textbf{26} (2001), no.~1-2, 101--174.

\bibitem[Stu06]{sturm2006geometry}
K.-T.~Sturm, \emph{On the geometry of metric measure spaces. I}, Acta Math. \textbf{196} (2006), no.~1, 65--131.

\bibitem[Stu26]{sturm2026sectional}
K.-T.~Sturm, \emph{Sectional curvature for Kantorovich-Wasserstein and Hellinger-Kantorovich geometries}, arXiv:2606.14318, 2026.

\bibitem[Vil09]{villani2009optimal}
C.~Villani, \emph{Optimal transport: old and new}, Grundlehren der mathematischen Wissenschaften, vol.~338, Springer-Verlag, Berlin, 2009.

\end{thebibliography}

\end{document}